\documentclass[a4paper,11pt]{amsart}
\usepackage[foot]{amsaddr}
\usepackage{geometry}
\usepackage{amsfonts}
\usepackage{amsthm}
\usepackage{amssymb}
\usepackage{amsmath}
\usepackage{theoremref}
\usepackage{enumerate}
\usepackage{bbm}
\usepackage{bm}
\usepackage{mathtools}
\usepackage{a4wide}
\usepackage{xcolor}
\usepackage{lipsum}

\makeatletter
\newcommand{\proofpart}[2]{%
  \par
  \addvspace{\medskipamount}%
  \noindent\emph{Step #1: #2}\par\nobreak
  \addvspace{\smallskipamount}%
  \@afterheading
}
\makeatother

\DeclarePairedDelimiter\abs{\lvert}{\rvert}%
\DeclarePairedDelimiter\norm{\lVert}{\rVert}%

\makeatletter
\let\oldabs\abs
\def\abs{\@ifstar{\oldabs}{\oldabs*}}
\let\oldnorm\norm
\def\norm{\@ifstar{\oldnorm}{\oldnorm*}}
\makeatother

\makeatletter
\g@addto@macro\bfseries{\boldmath}
\makeatother

\newcommand{\A}{\mathcal{A}}
\newcommand{\C}{\mathbb{C}}
\newcommand{\Dy}{\mathcal{D}}

\newcommand{\K}{\mathcal{K}}
\newcommand{\G}{\mathcal{G}}

\newcommand{\T}{\partial\mathbb{D}}
\newcommand{\Ca}{\mathcal{C}}
\newcommand{\z}{\zeta}

\newcommand{\conj}[1]{\overline{#1}}
\newcommand{\D}{\mathbb{D}}
\newcommand{\B}{\mathcal{B}}
\newcommand{\Po}{\mathcal{P}}

\newcommand{\cD}{\conj{\mathbb{D}}}

\newcommand{\dist}[2]{\text{dist}( #1, #2 ) }

\renewcommand{\Dy}{\mathcal{D}}

\renewcommand\Re{\operatorname{Re}}

\newtheorem{thm}{Theorem}[section]
\newtheorem{lemma}[thm]{Lemma}

\newtheorem{cor}[thm]{Corollary}
\newtheorem{prop}[thm]{Proposition}
\theoremstyle{definition}

\theoremstyle{definition}

\newcommand{\Addresses}{{% additional braces for segregating \footnotesize
		\bigskip
		\footnotesize
		
		Adem Limani, \\ \textsc{Departament de Matem\`{a}tiques \\ Universitat Aut\`{o}noma de Barcelona \\
		08193 Barcelona}\\
		\texttt{ademlimani@mat.uab.cat}
		
		\medskip
		Artur Nicolau, \\ \textsc{Departament de Matem\`{a}tiques \\ Universitat Aut\`{o}noma de Barcelona \\
        Centre de Recerca Matem\`atica \\
		08193	Barcelona}\\
		\texttt{artur.nicolau@uab.cat}
			
	}}

\begin{document}
\title{\textbf{Shift invariant subspaces in the Bloch space}}

\author{Adem Limani} 
\address{Departament de Matem\`atiques, Universitat Autònoma de Barcelona, Barcelona 08193}
\email{ademlimani@mat.uab.cat}

\author{Artur Nicolau}

\address{Centre de Recerca Matem\`atica,
Barcelona 08193}
\email{artur.nicolau@uab.cat}

\date{\today}

\maketitle

\begin{abstract}
\noindent
We consider weak-star closed invariant subspaces of the shift operator in the classical Bloch space. We prove that any bounded analytic function decomposes into two factors, one which is cyclic and another one generating a proper shift invariant subspace, satisfying a permanence property, which in a certain way is opposite to cyclicity. Singular inner functions play the crucial role in this decomposition. We show in several different ways that
the description of shift invariant subspaces generated by inner functions in the Bloch spaces deviates substantially from the corresponding description in the Bergman spaces, provided by the celebrated Korenblum and Roberts Theorem. %on the description of shift invariant subspaces generated by inner functions %describing cyclicity and the permanence property in the Bergman spaces deviates substantially from the setting of the Bloch space. 
Furthermore, the relationship between invertibility and cyclicity is also investigated and we provide an invertible function in the Bloch space which is not cyclic therein. Our results answer several open questions stated in the early nineties.  

% We investigate $M_z$-invariant subspaces generated by inner functions in the classical Bloch space. We prove that any singular inner function decomposes into two factors, one which generates a proper $M_z$-invariant subspace satisfying a certain permanence property and the other which is cyclic. Furthermore, we illustrate different concepts giving rise to singular factors satisfying the permanence property, connected to certain boundary zero sets and Sobolev functions in the associated model spaces. Moreover, we also provide a new sufficient condition for cyclic vectors in $\B$. Our developments enables us to answer a couple of questions on cyclic vectors in the Bloch space, left open in 1991 from a series of work by Anderson, Brown, Fernandez and Shields.
\end{abstract}

\section{Introduction} \label{INTROsec} 
\noindent
Let $\D$ denote the open unit disc in the complex plane $\C$ and let $\B$ be the classical Bloch space of analytic functions $f$ in $\D$ satisfying
\[
\norm{f}_{\B}:= \abs{f(0)}+ \sup_{z\in \D} (1-|z|)|f'(z)| < \infty.
\]
Taking the closure of analytic polynomials in the norm above we obtain the so-called little Bloch space $\B_0$, which is a separable Banach subspace of $\B$ consisting of functions $f$ satisfying
\[
\lim_{|z|\to 1^{-}} (1-|z|)f'(z) =0.
\]
Let $H^\infty$ designate the Banach space of bounded analytic functions $f$ in $\D$ equipped with the usual norm $\|f\|_\infty = \sup \{|f(z)| : z \in \D \}$, and recall that $H^\infty \subset \B$ by Schwarz lemma. Functions in the Bloch space are intimately related to conformal mappings, and as such they are regarded as a crucial objects in geometric function theory, see \cite{garnett2005harmonic} for further details. The Bloch space may also be regarded as the natural substitute for $H^\infty$ in the limiting case when $p \to \infty$ of the Bergman spaces $B_p$ of analytic functions $f$ in $\D$ satisfying
\[
\int_{\D} |f(z)|^p d A(z) < \infty,
\]
where $dA$ denotes the area measure on $\D$. See \cite{hedenmalmbergmanspaces} for a detailed treatment. For $1 \leq p < \infty$, we denote by $W^p$ the Sobolev space of analytic functions $g$ on $\D$ satisfying 
\[
\norm{g}^p_{W^p} := |g(0)|^p +  \int_{\D} \abs{g'(z)}^p dA(z)  < \infty.
\] 
Let $dm$ denote the Lebesgue measure on $\T$. It is well-known that the dual space of $W^1$ can be identified with $\B$ considered in the Cauchy pairing. More precisely, for any $f\in \B$ and $g\in W^1$, the limit
\[
\ell_f(g) := \lim_{r\to 1-} \int_{\T} f(r\zeta) \conj{g(r\zeta)} dm(\zeta) 
\]
exists and $f$ induces a unique bounded linear functional on $W^1$. Similarly, the dual of $\B_0$ can be identified with $W^1$, see for instance \cite{pommerenkebloch}. Let $M_z$ denote the multiplication operator by the independent variable $z$, that is, the linear operator $M_zf(z) = zf(z)$, $f \in \B$. Our main purpose is to investigate a certain class of $M_z$-invariant subspaces in $\B$, and to adjust for the fact that $\B$ is not separable, we shall instead consider weak-star closed subspaces in $\B$ which are $M_z$-invariant. Given a function $f\in \B$, we shall denote by $\left[ f\right]_{\B}$ the smallest weak-star closed $M_z$-invariant subspace containing $f$, that is, $\left[ f\right]_{\B}$ is the weak-star closure of polynomial multiples of $f$. If $f\in \B_0$, then $\left[f\right]_{\B}$ is actually the norm-closure of polynomial multiples of $f$ by Mazur's Theorem. A function $f  \in \B$ is called \emph{cyclic} in $\B$ if $\left[f\right]_{\B} = \B$. Since the set of analytic polynomials is weak-star dense in $\B$, we have that $f$ is cyclic in $\B$ if and only if $1\in \left[ f\right]_{\B}$. It is worth mentioning that $\left[ f\right]_{\B}$ may be much larger than the subspace of weak-star limits of sequences of polynomial multiples of $f$, as we shall demonstrate in \thref{THM:Seqcyc}. A function $f \in H^\infty$ is declared to satisfy the \emph{permanence property} (in $\B$) if the corresponding $M_z$-invariant subspace satisfies $\left[ f \right]_{\B} \cap H^\infty = f H^\infty$. The permanence property should be understood as an antithesis to cyclicity, in the sense that any bounded analytic function in $\left[f\right]_{\B}$ is divisible by $f$. %Thus such $M_z$-invariant subspaces $\left[f\right]_{\B}$ are of Beurling-type. Note that for the permanence property only the containment $\subseteq$ is important as the reverse inclusion $\supseteq$ vacuously holds. 

% As a pre-cautionary remark, we remind the reader that although the Krein-Smulian Theorem ensures that a convex set $S$ is weak-star closed in $\B$ if and only if it its weak-star sequentially closed, it does not mean that $\left[f\right]_{\B}$ is the weak-star sequential closure of the smallest $M_z$-invariant subspace containing $f$. We shall see that this observation elucidates a substantial barrier for understanding $M_z$-invariant subspaces in $\B$.  However, for $f\in \B_0$ the weak-star closed $M_z$-invariant subspace generated by $f$ agrees with the norm closure in $\B_0$, hence we may denote it by $\left[f \right]_{\B_0}$. This implies that $F \in \left[f \right]_{\B_0} $
% if and only if there exists a sequence of analytic polynomials $\{Q_n\}_n$ such that 
% \[
% \lim_n \norm{f Q_n -F}_{\B_0} = 0.
% \]
% Rephrasing the conditions for weak-star convergence of sequences in $\B$, a sufficient condition for $f \in \B$ to be weak-star cyclic is the existence of a sequence of analytic polynomials $\{Q_n\}_n$ satisfying the following properties:
% \begin{enumerate}
%     \item[(i)] $\sup_n \norm{fQ_n}_{\B} < \infty$,
%     \item[(ii)] $f(z)Q_n(z) \to 1$ for every $z\in \D$, as $n\to \infty$.
% \end{enumerate}

In this note, we shall mainly be concerned with $M_z$-invariant subspaces in $\B$ generated by bounded analytic functions. Recall that any $f \in H^\infty$ can be factored as $f=cFBS_\mu$ with unimodular constant $c$, where $F$ denotes the outer factor of $f$ defined by 
\[
F(z) = \exp \left( \int_{\T} \frac{\zeta+z}{\zeta-z} \log|f(\zeta)| d m (\zeta) \right), \qquad z \in \D , 
\]
the function $S_\mu$ is the singular inner factor of $f$ given by
\[
S_\mu(z) = \exp \left(- \int_{\T} \frac{\zeta+z}{\zeta-z} d\mu(\zeta) \right) , \qquad z \in \D ,
\]
where $\mu$ is a positive finite Borel measure on $\T$ which is singular with respect $dm$, while $B$ denotes the Blaschke product encoding the zeros of $f$ on $\D$, that is, 
\[
B(z)=\prod_{a: f(a)=0}  \frac{\abs{a}}{a} \frac{a-z}{1-\conj{a}z}, \qquad z\in \D.
\]
Functions in $H^\infty$ having radial limits of modulus one at almost every point of the unit circle are called \emph{inner functions}, and they are always of the form $\Theta = B S_\mu$. Given $f \in H^\infty$ with inner-outer factorization $f=FBS_\mu$, we denote by $\nu_f$ its associated Herglotz-Nevanlinna measure defined by
\[
d \nu_f = -\log|f(\zeta)| d m (\zeta) +  d \mu (\zeta) +\sum_{z: f(z)=0} (1-|z|^2) \delta_{z},
\]
where $\delta_z$ denotes the Dirac point mass measure at $z\in \D$. Note that $\nu_f$ is a positive measure whenever $\norm{f}_{\infty}\leq 1$. It was established in \cite{brown1991multipliers} that classical outer functions belonging to $\B$ must be cyclic, and the existence of a cyclic singular inner function in $\B_0$ was proved in \cite{anderson1991inner}. Our first result provides a structural theorem for $M_z$-invariant subspaces generated by bounded analytic functions, where singular inner functions play the decisive role. 

\begin{thm}\thlabel{THM:M_zLat} 
Let $f=FS_\mu B$ be in $H^\infty$, where $F$ is outer, $S_\mu$ is singular inner and $B$ is a Blaschke product. Then there exists a unique (up to sets of $\mu$-measure zero) decomposition $\mu = \mu_P + \mu_C$ of $\mu$, where $\mu_P, \mu_C$ are mutually singular positive measures giving rise to the following dichotomy.
\begin{enumerate}
    \item[(i)] The inner function $\Theta_0:=BS_{\mu_P}$ generates a proper $M_z$-invariant subspace in $\B$ satisfying the permanence property, that is, $ \left[ \Theta_0 \right]_{\B} \cap H^\infty = \Theta_0 H^\infty$.
    
    \item[(ii)] $FS_{\mu_C}$ is cyclic in $\B$.
\end{enumerate}
    
\end{thm}
\thref{THM:M_zLat} should be viewed as a Bloch space version of Theorem 1 in \cite{roberts1985cyclic} by Roberts in the context of Bergman spaces. Our result shows that in order to understand $M_z$-invariant subspaces in $\B$ generated by functions in $H^\infty$, it suffices to study the permanence property and cyclicity of singular inner functions. It turns out that these two notions have natural interpretations in the context of model spaces. For $0<p<\infty$ let $H^p $ denote the classical Hardy space of analytic functions $f$ in $\D$ such that
\[
\|f\|_p^p = \sup_{0<r <1} \int_{\T} |f(r \zeta)|^p dm(\zeta) < \infty. 
\]
We recall that given an inner function $\Theta$, the associated model space $K_\Theta$ is defined as $K_\Theta = H^2 \ominus \Theta H^2,$ that is, the orthogonal complement of $\Theta H^2$ in the classical Hardy space $H^2$. It follows from the celebrated Beurling Theorem that the model spaces are the only invariant subspaces for the adjoint $(M_z)^*$ viewed as an operator on $H^2$. A deep result of Aleksandrov says that functions with continuous extensions to $\T$ in any model space are always dense therein, see \cite{aleksandrovinv} and also \cite{dbrcont}, \cite{limani2022abstract} for recent generalizations. Recall that a closed set $E \subset \T$ of Lebesgue measure zero is said to be a Beurling-Carleson set if 
\[
\int_{\T} \log \text{dist} \left( \zeta, E\right) dm(\zeta) > -\infty.
\] 
It was recently established that $\bigcap_{p>1}W^p \cap K_{\Theta}$ is dense in $K_{\Theta}$ if and only if $\Theta = B S_\mu$ and $\mu$ is concentrated on a countable union of Beurling-Carleson sets \cite{limani2023model}, while $\cup_{p>1} W^p \cap K_{S_\mu} = \{0 \}$ if and only if $\mu$ does not charge any Beurling-Carleson set \cite{starinvsmooth}. We shall now consider what happens when $p=1$ in our next result. Recall that for an inner function $\Theta$, one defines the associated model space $K^1_\Theta := H^1 \cap \Theta \conj{zH^1}$, interpreted in the sense of boundary values on $\T$. In other words, a function $f \in H^1$ belongs to $K^1_\Theta$ if there exists a function $g \in H^1$ with $g(0)=0$ such that $\Theta (\zeta) \overline{f(\zeta)} = g(\zeta)$ for $m$-a.e. $\zeta \in \T$. 

\begin{thm}\thlabel{THM:Modelspaces}
Let $\mu$ be a positive, finite, Borel, singular measure in $\T$ and let $S_\mu$ be the corresponding singular inner function. Then
\begin{enumerate}
    \item[(i)] $S_\mu$ is cyclic in $\B$ if and only if $K^1_{S_\mu} \cap W^1 = \{0 \}$.
    \item[(ii)] The permanence property  $\left[ S_\mu \right]_{\B} \cap H^\infty = S_\mu H^\infty$ holds if and only if $K^1_{S_\mu} \cap W^1$ is dense in $K^1_{S_\mu}$. 
    \end{enumerate}
    
\end{thm}
Viewing $\B$ as a limiting case of the Bergman spaces, we shall now make a few comparisons. In the context of Bergman spaces, the celebrated Theorem of Korenblum and Roberts asserts that any singular measure $\mu$ uniquely decomposes as $\mu = \mu_{\Ca} + \mu_{\K}$, where $\mu_{\Ca}$ is concentrated on a countable union of Beurling-Carleson sets and gives rise to a proper $M_z$-invariant subspace generated by $S_{\mu_{\Ca}}$ satisfying the permanence property on the Bergman space, while $\mu_{\K}$ charges no Beurling-Carleson set and induces a cyclic vector $S_{\mu_{\K}}$ therein, see \cite{korenblum1981cyclic}, \cite{roberts1985cyclic}. We shall see that the situation in the Bloch space is very different and the results of Korenblum and Roberts do not carry over to this setting. Although we have not been able to find complete geometric descriptions of singular measures $\mu$ for which the corresponding singular inner functions $S_\mu$ are cyclic or satisfy the permanence property in $\B$, our progress seems to indicate that such an accomplishment is bound to be very difficult. Nevertheless, we are still able to provide several related conditions which allows us to answer various open questions and problems posed in the works of \cite{brown1991multipliers}, \cite{anderson1991inner}, \cite{pommerenkebloch} and \cite{starinvsmooth}. 

We declare a continuous non-decreasing and sub-additive function $w$ on $[0,1]$ with $w(0)=0$ to be a \emph{majorant}. A closed set $E \subset \T$ of Lebesgue measure zero is said to have finite $w$-entropy if
\[
\int_{\T} \log w \left( \text{dist}\left( \zeta, E \right) \right) dm(\zeta) >- \infty.
\]
Of course, when $w(t) = t^\alpha$ for some $0<\alpha <1$, one retains the classical Beurling-Carleson sets. Various descriptions of sets having finite $w$-entropy have recently been treated in \cite{ivrii2022beurling} and in \cite{limani2023mzinvariant}. Sets of finite $w$-entropy are precisely the boundary zero sets of analytic functions in $\D$ which extend continuously to $\T$ and whose modulus of continuity do not exceed $w$ on $\cD$, see \cite{carlesonuniqueness} for classical Beurling-Carleson sets and \cite{shirokov1982zero} for general majorants. As mentioned above, a singular inner function is cyclic in the Bergman space if and only if its associated singular measure does not charge any Beurling-Carleson set. It turns out that if a singular inner function is cyclic in the Bloch space, then its associated singular measure cannot charge a far wider range of sets.

%which cannot be charged by singular measures $\mu$ for cyclic functions $S_\mu$ in $\B$.

\begin{thm}\thlabel{THM:wsets} Let $w$ be a majorant. Assume that there exists $0<\gamma<1$ such that $w(t)/t^{\gamma}$ is non-increasing and that $w$ satisfies the Dini condition 
\[
\int_0^1 \frac{w(t)}{t}dt < \infty.
\]
If $\mu$ is concentrated on a countable union of sets having finite $w$-entropy, then the associated singular inner function $S_\mu$ satisfies the permanence property: $\left[ S_\mu \right]_{\B} \cap H^\infty = S_\mu H^\infty$. 

\end{thm}

We shall in fact deduce \thref{THM:wsets} from a slightly stronger statement in Section 3 (see \thref{THM:PPBMOAw} therein), which essentially asserts that inner factors of certain weighted BMOA spaces gives rise to the permanence property. In order for $S_\mu$ to be cyclic in $\B$, \thref{THM:wsets} says that $\mu$ cannot charge any set of finite $w$-entropy. Note that the class of eligible majorants $w$ for which our Theorem holds, goes far beyond classical Beurling-Carleson sets, as it also applies to slowly increasing majorants of type $w(t) = \log^{-\alpha}(e/t)$, with $\alpha>1$. The remark following Theorem 4 in \cite{anderson1991inner} also suggests that the above Dini condition on $w$ is essentially sharp.

Since any function cyclic in $\B$ must necessarily be cyclic in the Bergman spaces, we have that $S_\mu$ is cyclic in $\B$ implies that $\mu$ cannot charge Beurling-Carleson sets. The question whether the converse is true was raised in 1991 by Brown and Shields \cite{brown1991multipliers}. \thref{THM:wsets} above shows that this is not the case, but we shall below illustrate that this actually fails in a very strong sense. Before stating  our next result, note that if $S_\mu \in W^1$, then so are the reproducing kernels of $K_{S_\mu}$, and since they span a dense subspace of $K^1_{S_\mu}$, part (ii) of \thref{THM:Modelspaces} implies that the permanence property holds whenever $S_\mu \in W^1$. 
%As explained above the Korenblum-Roberts Theorem provides a description of cyclic singular inner functions in terms of the geometry of the carrier sets of $\mu$. Our next results illustrate that a similar description of the pieces $\mu_P, \mu_C$ appearing in \thref{THM:M_zLat}, in terms of carrier sets, or modulus of continuity is not feasible.

\begin{thm}\thlabel{THM:NoMOC} 
 Let $w$ be a majorant with $w(t)/t \to \infty$ as $t\to 0^+$.  Then there exists a singular probability measure $\mu = \mu (w)$ on $\T$ satisfying the following properties: 
\begin{enumerate}
    \item[(i)] For any arc $I \subset \T$, we have $\mu(I) \leq  w (|I|).$
    
    \item[(ii)] $S_\mu \in W^1$, hence it satisfies the permanence property, that is, $\left[ S_\mu \right]_{\B} \cap H^\infty = S_\mu H^\infty.$
    
\end{enumerate}
    
\end{thm}
If we take $w(t)=t^\alpha$ with $0 < \alpha <1$, then condition (i) implies that $\mu$ charges no Beurling-Carleson set \cite{shapiro1964weakly}, while (ii) shows that $S_\mu$ is not cyclic in $\B$. Hence not every cyclic singular inner function in the Bergman space is cyclic in $\B$, answering the problem left open in \cite{brown1991multipliers}. This brings to light a new remarkable discrepancy on the structure of $M_z$-invariant subspaces between the Hardy spaces and Bergman spaces. Although Beurling's Theorem on the Hardy spaces naturally carries over their corresponding limiting case $BMOA$ (see \cite{aleksandrovinv}), the Korenblum-Roberts Theorem on the Bergman spaces breaks down in $\B$. It turns out that the negative answer to the above mentioned question has an immediate consequence to another related problem in the context of model spaces, mentioned in \cite{starinvsmooth}. As a consequence of \thref{THM:Modelspaces} and \thref{THM:NoMOC} with $w(t)= t\log(e/t)$, we obtain the following conclusion.

\begin{cor}\thlabel{Modelneg} 
There exists a singular inner function $\Theta$ for which the following distinctive phenomena occur:
\begin{enumerate}
    \item[(i)] $K_\Theta \cap \bigcup_{p>1} W^p = \{0\}$.
    \item[(ii)] The reproducing kernels of $K_\Theta$ belong to $W^1$ and hence bounded analytic functions in $K_\Theta \cap W^1$ form a dense subset in $K_\Theta$.
    \end{enumerate}
    
\end{cor}
\noindent
It was also asked in \cite{anderson1991inner} whether the condition $|f(z)| \gtrsim {\log}^{-1} (e/(1-|z|))$ for any $z \in \D$ of a Bloch function $f$, ensures that $f$ is cyclic in $\B$. Our \thref{THM:NoMOC} shows that this is not the case, answering Problem 1 in \cite{anderson1991inner} in the negative. In fact, given any majorant $w$ with $w(t)/t \to \infty$ as $t \to 0^+$, \thref{THM:NoMOC} provides a positive singular measure $\mu$ such that $S_\mu$ is not cyclic in $\B$ while condition (i) readily translates to the estimate 
\[
\abs{S_\mu(z)} \gtrsim \exp \left( - c\frac{w(1-|z|)}{1-|z|} \right), \qquad z \in \D , 
\]
for some numerical constant $c>0$. We conclude that there cannot be any bound from below of a Bloch function which ensures it to be cyclic in $\B$. This pathological behavior is very different from the context of Bergman-type spaces, see \cite{borichev1996estimates}. As previously mentioned, $S_\mu$ satisfies the permanence property in Bergman spaces if and only if $\mu$ is concentrated on a countable union of Beurling-Carleson sets. Our next result says that the situation in the Bloch space is completely different and no condition on the support of $\mu$ alone can describe singular inner functions satisfying the permanence property in $\B$.

\begin{thm}\thlabel{THM:Nosupp}
Let $E \subset \T$ be a closed set of Lebesgue measure zero. Then there exists a singular probability measure $\mu$ supported on $E$, such that 
\begin{itemize}
    \item[(i)] $\mu(I) \geq |I|$, for any dyadic arc $I$ with $I \cap E \neq \emptyset$, 
    \item[(ii)] $S_\mu \in W^1$, hence it satisfies the permanence property, that is, $\left[ S_\mu \right]_{\B} \cap H^\infty = S_\mu H^\infty.$

\end{itemize}
\end{thm}
In order to construct a cyclic singular inner function in $\B$, the authors in \cite{anderson1991inner} provided a sufficient condition for functions in $\B$ to be cyclic. Our next result gives a  different sufficient condition which in a certain sense generalizes the above mentioned result.

\begin{thm}\thlabel{THM:SuffcycB} 
Let $f \in H^\infty$ be zero-free in $\D$ with $\norm{f}_{H^\infty}\leq 1$, and let $\nu=\nu_f$ denote its associated Herglotz-Nevanlinna measure. Assume there exists a constant $C= C(f)>0$, such that for any pair of contiguous arcs $I,I' \subset \T$ of same length $|I|= |I'|$, we have
\begin{equation}\label{expzyginfcond}
    \abs{ \frac{\nu(I)}{|I|} -\frac{\nu(I')}{|I'|} } \leq C \inf \exp \left( - \frac{\nu(J)}{|J|} \right), 
\end{equation}
where the infimum is taken over all arcs $J \subset \T$ with $J\supset I \cup I'$.  Then $f$ is cyclic in $\B$.

% (b) The function $1/f \in \B$ if and only if there exists a constant $C=C(f)>0$ such that for any pair of contiguous arcs $I, I' \subset \T$ of the same length, we have
% \begin{equation}\label{expzyginfcond2}
%     \abs{ \frac{\nu(I)}{|I|} -\frac{\nu(I')}{|I'|} } \leq C % \exp \left( - \frac{\nu(I)}{|I|} \right). 
% \end{equation}

\end{thm} 
Note that condition \eqref{expzyginfcond} clearly implies that $\nu$ is a Zygmund measure on $\T$, that is, the expression in the left hand side of \eqref{expzyginfcond} is uniformly bounded over the collection of pairs of contiguous arcs of the same length. However, condition \eqref{expzyginfcond} is considerably stronger as it asserts that if the density of the measure over a an arc is large, then the oscillation of the densities of the measure on subarcs must be substantially smaller. It is worth mentioning that the construction in \cite{anderson1991inner} of a singular measure $\mu$ for which $S_\mu$ is cyclic in $\B$ actually implies that $\mu$ satisfies the stronger condition
\[
\sup_{|I|\leq \delta} \abs{ \frac{\mu(I)}{|I|} -\frac{\mu(I')}{|I'|} } \leq C \inf_{|J| \geq \delta} \exp \left( - \frac{\mu(J)}{|J|} \right),
\]
which readily implies \eqref{expzyginfcond}. It turns out that if one removes the infimum in condition \eqref{expzyginfcond}, then one retains a description of bounded analytic functions which are invertible in $\B$. 

\begin{thm}\thlabel{THM:Invertibility B} 
Let $f \in H^\infty$ be zero-free in $\D$ with $\norm{f}_{H^\infty}\leq 1$, and let $\nu= \nu_f$ denote its associated Herglotz-Nevanlinna measure. Then $1/f \in \B$ if and only if there exists a constant $C=C(f)>0$ such that for any pair of contiguous arcs $I,I' \subset \T$ of the same length, we have
\begin{equation}\label{expzyginfcond2}
    \abs{ \frac{\nu(I)}{|I|} -\frac{\nu(I')}{|I'|} } \leq C \exp \left( - \frac{\nu(I)}{|I|} \right). 
\end{equation}
\end{thm}
For a description of elements in $H^\infty \cap \B_0$ of similar flavour, we refer the reader to the work of Bishop in \cite{bishop1990bounded}. 
%for a similar description of elements in $H^\infty \cap \B_0$. In light of \thref{THM:SuffcycB} and \thref{THM:Invertibility B}, it is tempting to ask whether the notions of invertibility and cyclicity in $\B$ of bounded analytic functions $f$ are the same. 
We now turn our attention to the problem of determining the relationship between invertibility and cyclicity, which is a vastly investigated topic in various spaces of analytic functions. It it is well-known that in the setting of commutative Banach algebras of analytic functions with units, the notions of cyclicity and invertibility are equivalent, while moving towards classes of analytic function for which the algebra property fails, such as the classical Dirichlet spaces and the Hardy spaces, invertibility is strictly stronger than cyclicity, see \cite{brown1990invertible}. However, the situation for Bergman spaces attracted considerable amount of attention and remained open for quite some time until it was resolved by Borichev and Hedenmalm in \cite{borichevhedenmalmcyclicity}. In their deep work, they constructed invertible functions in the Bergman Spaces, which are not cyclic therein, showing that the two notions are quite different in that setting. Our next result asserts that a similar phenomenon occurs in the Bloch space. This answers Problem 2 from \cite{anderson1991inner} in the negative.

\begin{thm}\thlabel{THM:CycvsInv} 
There exists $f \in \B$ with $1/f \in \B$ such that $f$ is not cyclic in $\B$.
\end{thm}
\noindent
The crucial property of a function $f$ as above is that $\abs{f'(z)}$ is as huge as possible for a considerably large set of points $z\in \D$. This aligns well with the principal philosophy surrounding the work in \cite{borichevhedenmalmcyclicity}, that an impediment to cyclicity of a Bergman space function is that the function enjoys maximal growth at a "massive" set in $\D$.

% \begin{thm}\thlabel{THM:CycvsInv} 
% There exists a singular inner function $S_\mu$ satisfying $1/S_\mu \in \B$ and for which precisely one of the following phenomenons occur.
% \begin{enumerate}
%     \item[(i)] Either $S_\mu$ is not cyclic in $\B$,
%     \item[(ii)] or, $S_\mu$ is not cyclic in $\B$, yet the set of all weak-star sequential limits in % $\B$ of functions of the form $S_\mu Q$, where $Q$ is an analytic polynomial, does not contain all of $\B$.
% \end{enumerate}
% \end{thm}
% \noindent
% In case $(i)$ occurs, this answers yet another question left open in \cite{anderson1991inner} (see problem 2). However, if $(ii)$ holds, then this means that sequential limits are not sufficient for (weak-star) cyclicity in $\B$, not even for singular inner functions. In fact, we shall provide a different argument asserting this claim, which should be viewed as a Bloch space version of Theorem 1.2 appearing in \cite{borichev1995cyclicity}. Recall that classical outer functions in $\B$ are always cyclic therein.
%We shall deduce the occurrence of a similar phenomena in the Bloch space setting, which answers yet another question left open in \cite{anderson1991inner} (see problem 2). 
%\begin{thm}\thlabel{THM:CycvsInv} There exists a singular inner function $S_\mu$ for which $1/S_\mu$ belongs to $\B$, yet $S_\mu$ is not cyclic in $\B$.
%\end{thm}
Our final result should be viewed as a pre-cautionary warning when considering weak-star closures of $M_z$-invariant subspaces generated by Bloch functions. In fact it may be viewed as a Bloch space version of Theorem 1.2 by Borichev and Hedenmalm in \cite{borichev1995cyclicity}. It asserts that sequential limits are not enough in order to capture the behavior of $M_z$-invariant subspaces in $\B$. The reader should bear in mind that classical outer functions in $\B$ are always cyclic therein.

\begin{thm} \thlabel{THM:Seqcyc}
There exists an outer function $f \in \B$ with $1/f \in H^\infty$, such that the set of all weak-star sequential limits in $\B$ of functions of the type $f(z)Q(z)$, where $Q$ is an analytic polynomial, is a proper subspace of $\B$.
%  which fail to have finite radial limits at a set of positive Lebesgue measure on $\T$.
\end{thm} 

The paper is organized as follows. Section \ref{SEC:Model} is devoted to establish \thref{THM:M_zLat} and \thref{THM:Modelspaces} in the more general framework of so-called regular spaces. In Section \ref{SEC:BZSets}, we deduce \thref{THM:wsets} on the the permanence property induced by certain boundary zero sets, while Section \ref{4} contains the proofs of \thref{THM:NoMOC} and \thref{THM:Nosupp}, which illustrate that the results of Korenblum and Roberts in Bergman spaces do not extend to the Bloch space. %and \thref{Modelneg}. 
Finally, Section \ref{SEC:SuffcycB1} is devoted to establishing our results on the theme of invertibility versus cyclicity, containing \thref{THM:SuffcycB}, \thref{THM:Invertibility B},  \thref{THM:CycvsInv} and \thref{THM:Seqcyc}.

\subsection*{Acknowledgements} 
The first author was supported by the Knut \& Alice Wallenberg Foundation, grant no. 2021.0294.
The second author was supported in part by the Generalitat de Catalunya (grant 2021 SGR 00071), the Spanish Ministerio de Ciencia e Innovaci\'on (project PID2021-123151NB-I00) and the Spanish Research Agency through the Mar\'ia de Maeztu Program (CEX2020-001084-M).  The authors would like to thank Konstantin Dyakonov, Oleg Ivrii and Bartosz Malman for fruitful discussions.

%Denote by $\left[f\right]_{\B}$ the weak-star closure of analytic polynomial multiples of $f\in \B$. Since the Cauchy pre-dual of $\B$ is separable the weak-star topology on $\B$ is locally metrizable, and hence by the Banach-Alaoglu theorem, any norm-closed ball of fixed radius in $\B$ is weak-star sequentially compact. These observations in conjunction with the Krein-Smulian theorem implies that any convex subset $Y$ of $\B$ is weak-star closed if and only if it is weak-star sequentially closed. %In particular, this means that $g\in \left[ f\right]_{\B}$ if and only if there exists a sequence $\{Q_n\}_n$ of analytic polynomials such that $fQ_n$ converges to $g$ in the weak-star topology on $\B$. 

\section{Beurling-type theorems and model spaces}
\label{SEC:Model}
In this section, we shall keep a broad point of view and illustrate a close relationship between Beurling-type theorems on $M_z$-invariant subspaces generated by inner functions $\Theta$ and the ample of functions with sufficiently regular boundary values belonging the model space $K_\Theta$.

\subsection{Regular spaces}
A Banach space $X$ of analytic functions in $\D$ will be referred to as a \emph{regular space} if the following three properties hold: 
\begin{itemize}
    \item [(i)] $X \subseteq H^1$.
    \item [(ii)] The set of analytic polynomials and the set of functions analytic in a neighborhood of $\cD$ are dense in $X$.
    \item[(iii)] For any $\ell$ in the dual space $X^*$, we have $\limsup_{n\to \infty} \, \abs{\ell(z^n)}^{1/n} \leq 1.$
\end{itemize}
Assumption $(i)$ ensures that the boundary values of functions in $X$ are integrable on $\T$ and by the closed graph theorem, there exists a constant $C>0$ such that $\|f\|_1 \leq C \|f\|_X $ for any $f \in X$. Assumption $(ii)$ implies that $X$ is a separable Banach space, thus Helly's selection theorem ensures that any closed and bounded set in the norm of its Banach space dual $X^*$ is sequentially compact wrt to the weak-star topology. In conjunction with $(iii)$, we may substitute the abstract Banach space dual-pairing between $X, X^*$ with the more practical \emph{Cauchy dual} $X'$, considered via the $H^2$-pairing 
\[
\lim_{r\to 1-} \int_{\T} f(r\z) \conj{g(r\z)} dm(\z), \qquad f\in X, g\in X'.
\]
Note that condition $(i)$ also gives the containment $H^\infty \subset X'$. For our purposes, we shall equip $X'$ with the weak-star topology. Let us briefly verify that $X= W^1$ is indeed a regular space. Note that if $g \in W^1$, then  
\[
\int_{\T} |g(r \z) - g(0)| dm (\z) \leq \int_{\T} \int_0^r  |g'(s \z)| ds dm (\z) \leq  \norm{g}_{W^{1}}, \quad 0<r<1.
\]
Hence $W^1 \subset H^1$ as required. It is clear that analytic polynomials are dense in $W^1$. Finally since the Cauchy-dual relation $(W^1)' \cong \B$ holds, condition (iii) is also fulfilled as the Taylor coefficients of Bloch functions are well-known to be bounded, see \cite{pommerenkebloch}. We thus conclude that $X= W^1$ indeed is a regular space. More generally, classical examples of regular spaces are provided by the analytic Sobolev spaces $W^p$ with $p\geq 1$ and their counter-parts involving multiple derivatives.

\subsection{$M_z$-invariant subspaces in duals of Regular Spaces}

Let $X$ be a regular space and let $X'$ denote its Cauchy dual. In this subsection, we collect  several general and simple properties of $M_z$-invariant subspaces in $X'$, which will be important for our further developments. Given a function $f\in X'$, we denote by $\left[f\right]_{X'}$ the weak-star closure in $X'$ of the smallest $M_z$-invariant subspace containing $f$. We declare an element $f$ to be cyclic in $X'$ if $\left[f\right]_{X'} = X'$. Similarly, given an inner function $\Theta$, we say that $\Theta$ satisfies the \emph{permanence property} in $X'$ if 
\[
\left[\Theta \right]_{X'} \cap H^\infty \subseteq \Theta H^\infty.
\]
The permanence property was recently introduced in the papers \cite{limani2022abstract} and \cite{limani2023problem}, in more specific context than ours. We also mention that it implicitly appeared in the earlier work of J. Roberts \cite{roberts1985cyclic}, in the context of Bergman spaces. Let $X$ be a regular space, let $\Theta$ be an inner function and let $\Po$ be the set of analytic polynomials. Observe that $X\cap K_\Theta^1$ regarded as a subset of $X$ is precisely the pre-annihilator of $\Theta \Po \subset X'$, denoted by $\Theta \Po_{\perp}$. A standard argument involving the Hahn-Banach separation theorem gives the following equality of sets
\begin{equation} \label{HBw*}
(X\cap K_\Theta^1)^\perp = (\Theta \Po_{\perp})^\perp = \left[ \Theta \right]_{X'}.
\end{equation}
Since the Cauchy reproducing kernels are in $X$, weak-star convergence in $X'$ implies convergence on compact subsets of $\D$. Hence the permanence property always holds for Blaschke products. In what follows, we shall thus restrict our attention to the permanence property for singular inner functions. Recall that $H^\infty$ is the Cauchy dual of the Banach space $L^1 (dm) / H^{1}_0$, where $H^1_0$ denotes the subspace of functions in $H^1$ vanishing at the origin. Our first simple observation asserts that for regular spaces $X$, the weak-star topology in $X'$ is coarser than the weak-star topology in $H^\infty$. For the sake of future references, we phrase it below and remark that the proof readily follows from the assumption $X\subseteq H^1$ of regular spaces.

\begin{lemma} \thlabel{HooX'}

Let $X$ be a regular space. Then $H^\infty \subset X'$, and whenever $\{f_n\}_n$ converges to $f$ weak-star in $H^\infty$, $\{f_n\}_n$ converges to $f$ weak-star in $X'$.
\end{lemma}

Our next observation allows us to substitute multiples of analytic polynomials by multiples of $H^\infty$-functions when considering (weak-star) cyclicity of a bounded analytic function in $X'$. %Such a result is not feasible for general functions in $\B$, since the multiplier algebra of $\B$ is strictly smaller than $H^\infty$, see \cite{brown1991multipliers}.

\begin{lemma}\thlabel{bddcyc}
Let $f\in H^\infty$. Then the weak-star closure in $X'$ of $fH^\infty := \{fh : h \in H^\infty \}$ equals $\left[ f \right]_{X'}$.
\end{lemma}  
\begin{proof}
It suffices to verify that $fh \in \left[f\right]_{X'}$ for any $h\in H^\infty$. Let $\{Q_n\}_n$ be a sequence of polynomials which converge weak-star in $H^\infty$ to $h$. Then since $f\in H^\infty$, $Q_nf$ converges weak-star in $H^\infty$ to $fh$. The conclusion now follows from \thref{HooX'}.
%the weak-star closure in $\B$ of $fH^\infty$ is contained in $\left[ f \right]_{\B}$.
%The trivial inclusion $f \Po \subset f H^\infty$, where $\Po$ denotes the set of analytic polynomials clearly implies the containment of $\left[f \right]_{\B}$ in the weak-star closure of $fH^\infty$ in $\B$. For the converse, it suffices to show that $fh \in \left[f\right]_{\B}$ for any $h\in H^\infty$. 
\end{proof}

Our next results illustrates a certain ordering of shift invariant subspaces generated by bounded analytic functions.

\begin{prop}[Division principle] \thlabel{cycdivprinc}
Let $f,g \in H^\infty$ be such that $f/g \in H^\infty$. Then $\left[f\right]_{X'} \subseteq \left[g \right]_{X'}$. In particular, whenever $f$ is weak-star cyclic in $X'$, then so is $g$.
\end{prop}
\begin{proof}
Set $h= f/g \in H^\infty$ and note that it suffices to show that $f= hg \in \left[g\right]_{X'}$. This however follows from the same argument as in \thref{bddcyc}.
\end{proof}

For singular inner functions $S_\mu$, the division principle implies that whenever $S_\mu$ is cyclic in $X'$, then so is any divisor $S_{\mu_0}$ of $S_{\mu}$, where $0\leq \mu_0 \leq \mu$. In the opposite direction, a monotonicity principle for the permanence property of singular inner functions in $X'$ holds.

\begin{prop}[Monotonicity principle] \thlabel{Ppropmonprinc} Let $\{\mu_n\}_n$ be a sequence of positive, finite, Borel, singular measures with $\mu_n \leq \mu_{n+1}$, $n \geq 1$, such that $\mu_n \to \mu$ weak-star in the space of complex finite Borel measures on $\T$. Then if for any $n \geq 1$, the associated singular inner functions $S_{\mu_n}$ satisfies the permanence property in $X'$, then so does $S_{\mu}$.
    
\end{prop}

\begin{proof}
Note that by monotonicity each $S_{\mu_n}$ is a divisor of $S_{\mu}$ hence $K^1_{S_{\mu_n}} \subseteq K^1_{S_\mu}$ for each $n$. Since $S_{\mu_n}$ satisfies the permanence property for any $n \geq 1$, then according to Proposition \ref{cycdivprinc}, we obtain 
\[
\left[S_\mu\right]_{X'} \cap H^\infty \subseteq \left[S_{\mu_n}\right]_{X'} \cap H^\infty \subseteq S_{\mu_n}H^\infty, \qquad n \geq 1.
\]
Let $f \in \left[S_\mu\right]_{X'} \cap H^\infty$ be arbitrary and note that by the above containment, we can for any $n$ find a bounded analytic function $h_n$ such that $f= S_{\mu_n}h_n$. Now since $ \norm{h_n}_{\infty} = \norm{f}_{\infty},$ for $n \geq 1 $, 
% \[
% \sup_n \norm{h_n}_{H^\infty} = \norm{f}_{H^\infty},
% \]
we may extract a subsequence $\{h_{n_k}\}_k$ which converges pointwise in $\D$ to some $h \in H^\infty$. Since $S_{\mu_n}$ converges pointwise in $\D$ to $S_\mu$, we conclude that $f= S_{\mu}h$. Thus $S_\mu$ satisfies the permanence property in $X'$ as desired.
\end{proof}

Our final lemma asserts that $M_z$-invariant subspaces generated by bounded analytic functions are invariant under the multiplication of outer functions in $H^\infty$. In particular, outer functions in $H^\infty$ are always cyclic in $X'$. %This result was observed in the special case of of $X'= \B$ by Brown and Shields in \cite{brown1991multipliers}.

\begin{prop} \thlabel{Fcycstab}
Let $F \in H^\infty$ be an outer function and $g \in H^\infty$. Then $\left[ Fg \right]_{X'} = \left[ g \right]_{X'}$.
\end{prop}
\begin{proof}
The containment $Fg \in \left[ g\right]_{X'}$ follows from the proof of \thref{bddcyc} and gives the inclusion $\subseteq$. For the converse inclusion, we apply Theorem 7.4 in \cite{garnett} which yields a sequence of bounded analytic functions $\{F_n\}_n$ satisfying the following properties
\begin{enumerate}
    \item[(i)] $\abs{F_n(z)F(z)} \leq 1$, \qquad $z\in \D$, 
    \item[(ii)] $F_n(\zeta)F(\zeta)$ converges pointwise to $1$ for $m$-a.e $\zeta \in \T$. 
\end{enumerate}
%Then $FgF_n$ has uniformly bounded $H^\infty$-norm, thus also $X'$-norm, and $FgF_n$ converges to $g$ pointwise in $ \D$. 
Then $FgF_n$ converges weak star in $H^\infty$ to $g$. The result now follows from \thref{HooX'} and \thref{bddcyc}.
\end{proof}

\subsection{Proof of \thref{THM:Modelspaces} in the context of Regular Spaces}
Our main purpose is to establish the connection between the containment of a regular space $X$ in the model spaces $K_\Theta$ to $M_z$-invariant subspaces generated by $\Theta$ in $X'$ in the weak-star topology. However, the assumption $X \subseteq H^1$ actually makes the model spaces $K^1_\Theta$ more appropriate in this regard. We recall that the reproducing kernels of $K^1_{\Theta}$ are explicitly given by
\[
\kappa_{\Theta}(z,\lambda) := \frac{1-\conj{\Theta(\lambda)}\Theta(z)}{1-\conj{\lambda}z}, \qquad \lambda,z \in \D
\]
and their linear span forms a dense subset in $K^1_{\Theta}$. Our first observation is related to cyclicity of inner functions $\Theta$ in $X'$ and the existence of non-trivial functions in $X \cap K^1_\Theta$. We remark that similar results were known for a wide range of analytic Sobolev spaces $X$ in \cite{starinvsmooth} and more recently, it has appeared implicitly in a more general setting similar to ours in \cite{limani2022abstract}. 

\begin{prop}\thlabel{prop2.1} Let $X$ be a regular space. An inner function $\Theta$ is weak-star cyclic in $X'$ if and only if $X\cap K_\Theta^1$ contains no non-trivial function.
\end{prop}

\begin{proof}
Any $g \in X \cap K^1_\Theta$ is annihilated by polynomial multiples of $\Theta$, hence if we assume that $\Theta$ is weak-star cyclic in $X'$, then $g\equiv 0$. Conversely, if $\Theta$ is not weak-star cyclic in $X'$, then there exists a non-trivial $g\in X$ such that 
\[
\int_{\T} \conj{g(\z)}\Theta(\z) {\z}^n dm(\z) =0, \qquad n\geq 0.
\]
The F. and M. Riesz Theorem implies that there exists $ g_0 \in H^1$ with $g_0(0)=0$ such that $g= \Theta \conj{g_0}$ at almost every point of $\T$. Hence $g\in  H^1 \cap \Theta \conj{zH^1}= K_\Theta^1$.
\end{proof}

Next we shall relate the density of regular spaces in the model spaces $K^1_\Theta$ to the \emph{permanence property} of $\Theta$ in $X'$. The following result is quite similar to Theorem 1.3 in \cite{limani2022abstract}, but we shall need a more general version in our setting.

\begin{prop} \thlabel{Pprinc}
Let $X$ be a regular space and let $\Theta$ be an inner function. Then $X\cap K_\Theta^1$ is dense in $K_\Theta^1$ if and only if $\Theta$ satisfies the permanence property in $X'$.
\end{prop}

\begin{proof} 

Suppose that $X \cap K_\Theta^1$ is dense in $K_\Theta^1$ and let $f \in \left[ \Theta \right]_{X'}\cap H^\infty$. According to \eqref{HBw*}, $f$ must necessarily annihilate the family of reproducing kernels $\{k_\Theta(\cdot, \lambda)\}_{\lambda \in \D}$ of $K_\Theta^1$, that is
\[
f(\lambda) = \Theta(\lambda) \int_{\T} \frac{\conj{\Theta(\zeta)}f(\zeta)}{1-\conj{\zeta}\lambda} dm(\zeta), \qquad \lambda \in \D.
\]
Hence $ f \in \Theta H^2$, but since $f\in H^\infty$, we get that $f\in \Theta H^\infty$ as desired. Conversely, if the permanence property holds, then in fact we verify that the following stronger condition holds:
\begin{equation}\label{propH2}
\left[ \Theta \right]_{X'} \cap H^2 \subseteq \Theta H^2.
\end{equation}
To see this, note that if $f\in \left[ \Theta \right]_{X'} \cap H^2$, then $\left[f\right]_{X'} \subseteq \left[\Theta\right]_{X'}$. Let $f= F \Phi$, where $F$ is outer and $\Phi$ inner. It suffices to show that $\Theta$ divides $\Phi$ in $H^\infty$. For $n \geq 1$, consider the outer function $F_n$ defined as $\abs{F_n}= \min \left(n,\abs{F}\right)$ on $\T$ and note that if we set $f_n := F_n \Phi$, then $f_n/f \in H^\infty$ and hence by the division principle \thref{cycdivprinc}, we get
\[
\left[f_n \right]_{X'} \subseteq \left[f\right]_{X'} \subseteq \left[\Theta\right]_{X'}, \qquad n\geq 1.
\]
However since $f_n = F_n \Phi$ is bounded for each $n\geq 1$, \thref{Fcycstab} implies $\left[f_n\right]_{X'} = \left[\Phi\right]_{X'}$. We thus conclude that $\Phi \in \left[\Theta\right]_{X'} \cap H^\infty \subseteq \Theta H^\infty$ and hence the claim is proved. With this at hand, note that \eqref{propH2} in conjunction with \eqref{HBw*} may be rephrased as
\[
(X\cap K_\Theta^1)^\perp \cap H^2 \subseteq \Theta H^2.
\]
Now since the Cauchy dual $(K_\Theta^1)'$ is contained in $K^2_\Theta$ (see \cite{dyakonov2022interpolation}), we see that any annihilator of $X \cap K_\Theta^1$ is contained in $K^2_\Theta \cap \Theta H^2 = \{0\}$. This is enough to conclude that $X\cap K_\Theta^1$ is dense in $K_\Theta^1$.
\end{proof}

\begin{proof}[Proof of \thref{THM:Modelspaces}] 
Taking $X=W^1$, part (a) follows from \thref{prop2.1}, while $(b)$ follows from \thref{Pprinc}.
\end{proof}

\subsection{An abstract decomposition for singular measures} 
Inspired from the Korenblum-Roberts Theorem on $M_z$-invariant subspaces generated by inner functions, this section is devoted to extending their result to the context of Regular Spaces. Here we shall follow ideas developed by Roberts in \cite{roberts1985cyclic}, with certain required adaptations in order to deal with our general framework.

\begin{thm}[Abstract decomposition] \thlabel{Absdec}
Let $X$ be a regular space and let $S_\mu$ be a singular inner function. Then there exists a unique decomposition of $\mu= \mu_{P} + \mu_{C}$ (up to sets of $\mu$-measure zero) into mutually singular parts $\mu_{P}, \mu_{C}$ satisfying the following properties:
\begin{itemize}
    \item[(i)] $S_{\mu_P}$ has the permanence property in $X'$, that is,  
   $\left[S_{\mu_P} \right]_{X'} \cap H^\infty \subseteq S_{\mu_P} H^\infty$,
    \item[(ii)] $S_{\mu_C}$ is weak-star cyclic in $X'$, that is, $\left[S_{\mu_C} \right]_{X'} = X'$. 
    %\item[(iii)] The following holds $\left[ BS_\mu \right]_{X'} = \left[ B S_{\mu_{\Ca}} \right]_{X'}$.
\end{itemize}

\end{thm}

\begin{proof}
We shall divide the proof in the following two steps.
\proofpart{1}{The construction of $\mu_P$} 
Let $\mu \lvert_E$ denote the restriction of $\mu$ to the Borel set $E$. Consider the the following collection of Borel subsets of $\T$  
\[
\mathcal{M}(X',\mu):= \left\{E \subset \T: \, \text{E Borel set}, \, \, \, \mu (E) >0,  \left[ S_{\mu\lvert_E} \right]_{X'} \cap H^\infty \subseteq S_{\mu\lvert_E}H^\infty \right\}.
\]
If the set $\mathcal{M}(X',\mu)$ is empty, then we simply take $\mu_{C} = \mu$. For the being moment, we shall assume that the collection $\mathcal{M}(X',\mu)$ is non-empty, and our goal is to primarily establish that $\mathcal{M}(X',\mu)$ is closed under the formation of unions. To this end, pick two sets $E,F \in \mathcal{M}(X',\mu)$ which we may assume are not contained in one or another. Set $\nu := \mu \lvert_{E \cup F}$ and note that the containment of $K_{S_{\mu \lvert_E}}^1 \subset K_{S_\nu}^1$ in conjunction with \eqref{HBw*} implies
\[
\left[ S_{\nu} \right]_{X'} \cap H^\infty = \left( X \cap K_{S_{\nu}}^1 \right)^{\perp} \cap H^\infty \subseteq ( X \cap K_{S_{\mu \lvert_E}}^1 )^{\perp} \cap H^\infty =\left[ S_{\mu \lvert_E} \right]_{X'} \cap H^\infty \subseteq S_{\mu \lvert_E} H^\infty.
\]
Therefore, any $f \in \left[ S_{\nu} \right]_{X'} \cap H^\infty$ is divisible by $S_{\mu \lvert_E}$, and by switching the roles of $E$ and $F$, we deduce in a similar way that $f$ is also divisible by $S_{\mu \lvert_F}$. Altogether, we obtain $f/S_{\nu} \in H^\infty$ and therefore $E\cup F \in \mathcal{M}(X',\mu)$. Now consider the quantity
\[
\gamma(X',\mu) := \sup \left\{ \mu(E): \, E \in \mathcal{M}(X',\mu) \right\},
\]
and observe that since $\mathcal{M}(X',\mu)$ is closed under finite unions, we can find a sequence of Borel sets $\{E_n\}_n$ with $ E_n \in \mathcal{M}(X',\mu)$ and $E_n \subset E_{n+1}$ for any $n$, such that $\mu(E_n) \to \gamma(X',\mu)$ as $ n \to \infty$. Set $E:= \cup_n E_n$ and we claim that our candidates are given by $\mu_{P} = \mu \lvert_E$ and $\mu_{C} = \mu \lvert_{\T \setminus E}$. We first check that $E \in \mathcal{M}(X',\mu)$. To this end, let $f \in \left[ S_{\mu \lvert_E} \right]_{X'} \cap H^\infty $ and note that since each $E_n \in \mathcal{M}(X',\mu)$, an application of \eqref{HBw*} implies
%recall that a straightforward argument involving functional analysis and reproducing kernels shows that $\cup_n K_{S_{\mu \lvert_{E_n}}}$ is dense subsets of $K_{S_{\mu \lvert_E}}$ (monotonicity of the $E_n$'s is important). Using this in conjunction with  \thref{Pprinc}, we obtain
\[
\left[ S_{\mu \lvert_E} \right]_{X'} \cap H^\infty = (X \cap K_{S_{\mu \lvert_E}}^1)^\perp \cap H^\infty \subseteq (X \cap K_{S_{\mu \lvert_{E_n}}}^1)^\perp \cap H^\infty = \left[ S_{\mu \lvert_{E_n}} \right]_{X'} \cap H^\infty \subseteq S_{\mu \lvert_{E_n}} H^\infty,
\]
for all $n$. Hence for each $n$, there exists $\{h_n\}_n \subset H^\infty$ such that $f = S_{\mu \lvert_{E_n}} h_n$. But since $\norm{h_n}_{\infty}= \norm{f}_{\infty},$
we may extract a subsequence $\{h_{n_k}\}_k$ that converges pointwise in $\D$ to a certain function  $h \in H^\infty$. Now since the sequence of  sets $\{E_n\}_n$ is increasing, we actually obtain $f= S_{\mu_E}h$. This proves that $E \in \mathcal{M}(X',\mu)$ and consequently we have established claim $(i)$ with $\mu_{P} = \mu \lvert_E$. 
\proofpart{2}{Verifying cyclicity of $S_{\mu_C}$} 
We now proceed to establish claim $(ii)$ by showing that if $\mu_{C}:= \mu \lvert_{\T \setminus E}$ then $S_{\mu_{C}}$ is weak-star cyclic in $X'$. To this end, suppose that $S_{\mu_{C}}$ is not weak-star cyclic in $X'$, then there exists a non-trivial $g\in X$ such that 
\[
\int_{\T} \conj{g(\z)} S_{\mu_{C}}(\z) \z^n dm(\z) =0, \qquad n\geq 0.
\]
By the F. and M. Riesz Theorem, there exists a function $h_1 \in H^1$ with $h_1(0)=0$ such that 
\[
\conj{g(\z)} = \frac{h_1(\z)}{S_{\mu_{C}}(\z)} = \frac{h_2(\z)}{S_{\nu}(\z)}, \quad m-\text{a.e.} \, \, \,  \z \in \T.
\]
Here $h_2\in H^1$, $h_2 (0)=0$, $0 \leq \nu \leq \mu_{C}$ and the quotient $h_1/h_2$ is a singular inner function dividing $S_{\mu_{C}}$ (possibly a unimodular constant). Now if $\nu \equiv 0$, then $g \in H^1 \cap \conj{H^1}$ on $\T$ with $g(0)=0$, which implies $g\equiv 0$, thus contradicting the assumption that $g$ is non-trivial. Hence we may assume that $\nu \neq 0$, and our aim is to establish that $S_\nu$ satisfies the permanence property, which will contradict the maximality of $S_{\mu_{P}}$. To this end, note that
\[
\int_{\T}\conj{g(\z)} S_{\nu}(\z) \z^n dm(\z) = \int_{\T}h_2(\zeta) \z^{n} dm(\z) = 0, \qquad n\geq 0,
\]
hence by the F. and M. Riesz Theorem, we have $g \in K^1_{S_\nu}:= H^1 \cap S_{\nu} \conj{H^1_0}$. Now if $u \in \left[S_\nu \right]_{X'} \cap H^\infty$, then according to \eqref{HBw*} we actually get 
\[
\int_{\T}\conj{g(\z)} u(\z) \z^n dm(\z)=0, \qquad n\geq 0.
\]
Applying the F. and M. Riesz Theorem once again, we can find $h_3 \in H^1$ with $h_3(0)=0$ such that 
\[
\conj{g (\z)} = \frac{h_3 (\z)}{u (\z)} = \frac{h_2 (\z)}{S_{\nu} (\z)}, \quad m-\text{a.e.} \, \, \,  \z \in \T
\]
which implies that $u \in S_{\nu}H^\infty$. Hence $S_\nu$ satisfies the permanence property in $X'$. We have thus obtained the desired contradiction and it follows that $S_{\mu_{C}}$ is weak-star cyclic in $X'$. Obviously the same argument works in case the collection $\mathcal{M}(X',\mu)$ is void. The claim regarding uniqueness of the above decomposition of $\mu$ follows from maximality of $\mu_P$ with respect to the permanence property in $X'$. 

\end{proof}

\begin{proof}[Proof of \thref{THM:M_zLat}]
The following proof of \thref{THM:M_zLat} actually holds for any regular space $X$. Note that statement $(i)$ in \thref{Absdec} regarding the permanence property is easily shown to remain true if $S_{\mu_P}$ is multiplied by a Blaschke product. The claim on cyclicity in $(ii)$ of \thref{THM:M_zLat} follows from the cyclicity of $S_{\mu_C}$ in conjunction with \thref{Fcycstab}. This completes the proof. 

\end{proof}

As a consequence of \thref{Absdec} 
in conjunction with the results in Section 2.3, we conclude this section with the following description regular spaces in model spaces.

\begin{cor}\thlabel{strmodthm} Let $X$ be a regular space and $\Theta = B S_\mu$ be an inner function with corresponding Blaschke product $B$ and singular inner factor $S_\mu$. Consider the decomposition $\mu = \mu_P + \mu_C$ given in \thref{Absdec}. Then the following holds:
\begin{description}
    \item[\textit{Abundance of regular functions}] 
    $X\cap K^1_{B S_{\mu_P}}$ is a dense subspace of $K^1_{B S_{\mu_P}}$.
    \item[\textit{Absence of regular functions}] 
    $K^1_{S_{\mu_{C}}} \cap X = \{0\}$.
\end{description}

\end{cor}

\section{The permanence property and boundary zero sets} \label{SEC:BZSets} 
This section is devoted to showing that inner factors of certain weighted BMOA spaces give rise to the permanence property in $\B$, for which \thref{THM:wsets} will be derived as a corollary.

%on certain sets of finite $w$-entropy give rise to a permanence property in $\B$. In fact, we shall prove a slightly stronger result which essentially asserts that inner factors of certain weighted BMOA spaces give rise to a permanence property in $\B$.

\subsection{Inner factors in weighted BMOA}
%\subsection{Sets of finite $w$-entropy}
%We first introduce several notions which will be useful for our further developments here. 

Let $w$ be a majorant and denote by $BMO_w(\T)$ the space of Lebesgue integrable functions $h$ on $\T$ equipped with the semi-norm 
\[
\norm{h}_{BMO_w} := \sup_{I} \frac{1}{w(|I|)}  \frac{1}{|I|} \int_{I} |h - h_I| dm  < \infty.
\]
We denote by $C_{w}(\T)$ the space of continuous functions $h$ on $\T$ equipped with the semi-norm
\[
\norm{h}_{C_w}:= \sup_{\zeta\neq \xi} \frac{\abs{h(\zeta)-h(\xi)} }{w(\abs{\zeta-\xi})} < \infty.
\]
It is straightforward to verify the containment $C_w (\T) \subseteq BMO_w(\T)$ and it turns out that these spaces share very intimate features. If the majorant $w$ satisfies the Dini-condition 
\begin{equation}\label{Dini}
\int_0^1 \frac{w(t)}{t} dt < \infty,
\end{equation}
then functions in $BMO_w(\T)$ extend continuously to $\T$ with modulus of continuity not exceeding a constant multiple of $\int_{0}^\delta \frac{w(t)}{t}dt$. Conversely, if the Dini-condition \eqref{Dini} on $w$ fails then $BMO_w(\T)$ contains discontinuous and unbounded functions. 
% Furthermore, if $w$ is \emph{Dini-regular}, that is
% \[
% \sup_{0<\delta<1} \frac{1}{w(\delta)} % \int_0^{\delta} \frac{w(t)}{t} dt < \infty,
% \]
% then any function in $BMO_w(\T)$ coincides almost everywhere with a function in $C_w(\T)$. 
See \cite{spanne1965some} for details on these matters. Analytic counter-parts of the above defined spaces will be important, namely we consider the spaces $BMOA_w:= H^2 \cap BMO_w(\T)$ and $A_{w} := H^\infty \cap C_w(\T)$, hence the containment $A_w \subseteq BMOA_w$ holds. Moving forward, we shall now restrict our attention to majorants $w$ satisfying the following condition. There exists a constant $0<\gamma <1$ such that 
\begin{equation}\label{AA}
 w(t)/t^{\gamma} \text{ is non-increasing on } [0,1].
\end{equation}
We now state our main result in this section.

\begin{thm}\thlabel{THM:PPBMOAw} 
Let $w$ be a majorant satisfying the Dini condition \eqref{Dini} and condition \eqref{AA}. Then any non-trivial singular inner factor $S_\mu$ of a function in $BMOA_w$ satisfies the permanence property in $\B$, that is, $\left[ S_\mu \right]_{\B} \cap H^\infty = S_\mu H^\infty.$ In particular, no non-trivial singular inner divisor of a function in $BMOA_w$ can be cyclic in $\B$.
    
\end{thm}
An important class of singular inner factors of $BMOA_w$-functions are in fact provided by singular inner factors of $A_w$-functions, which turn out to be intimately connected to sets of finite $w$-entropy due to a deep result by Shirokov in \cite{shirokov1982zero}. For the moment being, we shall primarily establish \thref{THM:PPBMOAw} and then consequently derive \thref{THM:wsets} as a corollary. In our pursuit towards proving \thref{THM:PPBMOAw}, we shall need a couple of preparatory results. The first lemma can be found in a slightly greater generality in \cite{dyakonov1998multiplicative}, (see Proposition 2.6 therein).

\begin{lemma} \thlabel{LemBMOw} Let $w$ be a majorant satisfying condition \eqref{AA}. Then the $BMO_w$ semi-norm is equivalent to the Garsia-type semi-norm, that is, 
\[
\norm{h}_{BMO_w} \asymp \sup_{z \in \D} \frac{1}{w(1-|z|)} \int_{\T} \abs{ h(\zeta) - P(h)(z)} P_z(\zeta) dm(\zeta) 
\]
where $P(h)$ denotes the Poisson extension of $h$ on $\D$ and $P_z (\zeta) = (1-|z|^2)|1 - \overline{z} \zeta|^{-2}$, $z \in \D$, $ \zeta \in \T$, is the Poisson kernel on $\D$.
\end{lemma}

With \thref{LemBMOw} at hand, we may derive the following result on the Cauchy projection.
\begin{lemma} \thlabel{BMOwintoW11}

Let $w$ be a majorant satisfying the Dini condition \eqref{Dini} and condition \eqref{AA}. Then the Cauchy projection $P_+$ maps $BMO_w(\T)$ continuously into $W^1$.

\end{lemma}
\begin{proof}
Fix an arbitrary $h \in BMO_w(\T)$ and observe that 
\[
P_+(h)'(z) = \int_{\T} \frac{\conj{\zeta} h(\zeta)}{(1-\conj{\zeta}z)^2} dm(\zeta) = \int_{\T} \frac{\conj{\zeta} \left( h(\zeta)- P(h)(z) \right)}{(1-\conj{\zeta}z)^2} dm(\zeta), \qquad z\in \D.
\]
According to \thref{LemBMOw}, we have
\[
\abs{P_+(h)'(z)} \leq  \int_{\T} \frac{\abs{h(\zeta)- P(h)(z) } }{|1-\conj{\zeta}z|^2} dm(\zeta) \lesssim \norm{h}_{BMO_w} \frac {w(1-|z|)}{1-|z|}, \qquad z\in \D.
\]
It now readily follows that 
\[
\int_{\D} \abs{P_+(h)'(z)} dA(z) \lesssim  \norm{h}_{BMO_w}  \int_{0}^1 \frac {w(t)}{t} dt,
\]
which shows that $P_+ : BMO_w(\T) \to W^1$ continuously. 
\end{proof}

The following result on division and multiplication by inner functions on weighted $BMOA$ spaces due to K. Dyakonov in \cite{dyakonov1998multiplicative} plays a crucial role in our developments.

\begin{thm}[Dyakonov] \thlabel{THM:Dyakonov} Let $w$ be a majorant satisfying condition \eqref{AA}. Let $g\in BMOA_w$ and let $\Theta$ be an inner function. Then $g \conj{\Theta}$ belongs to $BMO_w(\T)$ if and only if $g\Theta$ belongs to $BMOA_w$.
    
\end{thm}

It follows from the work in \cite{dyakonov1998multiplicative} that if $w$ satisfies condition \eqref{AA}, then $BMOA_w$ enjoys the factorization property. More specifically, whenever $g\in BMOA_w$ and $\Theta$ is an inner function with $g/\Theta \in H^\infty$, then in fact $g/\Theta \in BMOA_w$. 

\begin{proof}[Proof of \thref{THM:PPBMOAw}]
Let $\Theta := S_\mu$ be a singular inner factor of a function $g \in BMOA_w$. By means of applying the factorization property of $BMOA_w$, we may without loss of generality assume that $g= G \Theta$, where $G$ is an outer function in $BMOA_w$. Now applying \thref{THM:Dyakonov} to the function $G$, we conclude that $G\conj{\Theta}$ belongs to $BMO_w(\T)$ and thus so does $\conj{G}\Theta$. Hence if $k_\lambda$ denotes the Cauchy kernel at the point $\lambda \in \D$, then the function $g_\lambda := \conj{G}\Theta k_\lambda$ also remains in $BMO_w(\T)$. Applying \thref{BMOwintoW11}, we conclude that $P_+(g_\lambda) \in W^1$. Given $H \in L^\infty (\T,dm)$, we denote by $T_{H}$ the Toeplitz operator with symbol $H$ defined as $T_{H}(f) := P_+(Hf)$. If $\kappa_{\Theta}(\cdot,\lambda)$ denotes the reproducing kernel of the model space $K_\Theta$ at the point $\lambda \in \D$, we have that
\[
T_{\conj{G}} \left( \kappa_{\Theta}(\cdot,\lambda) \right)(z) = T_{\conj{G}}(k_\lambda)  (z)- \conj{\Theta(\lambda)} P_+(g_{\lambda})(z), \quad z \in \D, 
\]
belongs to $W^1 \cap K_\Theta$. This follows from the following properties of Toeplitz operators with co-analytic symbols:
\begin{enumerate}
    \item[(a)] $T_{\conj{H}}: W^p \to W^p$ for any $H\in H^\infty$ and $1<p<\infty$.
    \item[(b)] $T_{\conj{H}}: K_{\Phi} \to K_{\Phi}$ for any $H\in H^\infty$ and any inner function $\Phi$.
\end{enumerate}
For instance, see \cite{cauchytransform}. Property $(a)$ is actually more than what is needed here, but it certainly guarantees that $T_{\conj{f}}(k_\lambda)$ belongs to $W^1$ for each $\lambda \in \D$, while property $(b)$ shows that $T_{\conj{G}} \left( \kappa_{\Theta}(\cdot,\lambda) \right)$ belongs to $K_{\Theta}$. Consequently, we conclude that $T_{\conj{G}} \left( \kappa_{\Theta}(\cdot,\lambda) \right)$ belong to $W^1 \cap K_\Theta$. We now proceed by verifying that the linear span $\mathcal{M}$ of the set $\{T_{\conj{G}} (\kappa_{\Theta}(\cdot, \lambda) )\}_{\lambda\in \D}$ is dense in $K_{\Theta}$. To this end, let $f\in K_{\Theta}$ be an element which annihilates $\mathcal{M}$. Then 
\[
0= \int_{\T} T_{\conj{G}} (\kappa_{\Theta}(\cdot, \lambda) )(\zeta) \conj{f(\zeta)} dm(\zeta) = \int_{\T}\kappa_{\Theta}(\zeta, \lambda) \conj{G(\zeta) f(\zeta)} dm(\zeta), \qquad \lambda \in \D.
\]
Since the linear span of the reproducing kernels are dense in $K_{\Theta}$, we conclude that there exists a function $h \in H^2$ such that $Gf = \Theta h$. Now recalling that $G$ was outer, we conclude that $f \in \Theta  H^2 \cap K_{\Theta} = \{0\}$, and thus $\mathcal{M}$ is dense in $K_{\Theta}$. Now since $K_{\Theta}$ is dense in $K^1_{\Theta}$, we also conclude that $\mathcal{M}$ is dense in $K^1_{\Theta}$. Recalling that $W^1$ is the Cauchy pre-dual of $\B$, it now follows from \thref{Pprinc} that $\left[\Theta \right]_{\B} \cap H^\infty \subseteq \Theta H^\infty$, hence the proof is complete.

\end{proof}

% \subsection{Sets of finite $w$-entropy}

We now turn our attention to the proof of \thref{THM:wsets}.

\begin{proof}[Proof of \thref{THM:wsets}]
We shall primarily assume that the positive finite singular measure $\mu$ on $\T$ is supported on a single (closed) set $E$ of finite $w$-entropy. According to Shirokov's Theorem in \cite{shirokov1982zero}, there exists an outer function $f \in A_w$ such that the product $fS_{\mu}$ belongs to $A_w \subset BMOA_w$. Hence by \thref{THM:PPBMOAw} $S_\mu$ satisfies the permanence property in $\B$. Now assume that $\mu$ is concentrated on countable union of sets $\{E_n\}_n$ having finite $w$-entropy, which we may assume are increasing: $E_n \subseteq E_{n+1} $ for any $n$. Consider $\mu_n := \mu \lvert_{E_n}$ and observe that $\mu_n$ converges to $\mu$ in the weak-star topology of finite complex Borel measures on $\T$. Since each $S_{\mu_n}$ satisfies the permanence property in $\B$, so does $S_\mu$ by the monotonicity principle in \thref{Ppropmonprinc}. The proof is now complete.

%Fix $0<\gamma < \min(1, 1/\alpha)$ and note that $w^{\gamma}$ is a majorant, hence according to Corollary 2.3 in \cite{berman1987moduli} one of the following situations occur: either $\sup_{0<t<1} w^{\gamma}(t)/t < \infty$, or $\lim_{t\to 0+} w^{\gamma}(t)/t = \infty$, that is, condition $(A)$ holds. In the former case, $E$ is then a classical Beurling-Carleson set and hence it satisfies the permanence property by \thref{PPW11} in conjunction with \thref{suffW11cond}. On the other hand, if $\lim_{t\to 0+} w^{\gamma}(t)/t = \infty$, then since 
%\[
%\int_{0}^1 \frac{w^{\alpha'}(t)}{t} dt <\infty.
%\]

%Observe that $E$ is a set of finite $w$-entropy if and only if it has finite $w^{\alpha}$-entropy for any $\alpha>0$, hence by means of substituting $w$ with $w^\alpha$, we may w.log assume that
%\[
%\int_0^1 \frac{w(t)}{t} dt < \infty.
%\]
 %On the other hand, if $\lim_{t\to 0+} w(t)/t = \infty$, then since 
%\[
%\int_{0}^1 \frac{w^{\alpha'}(t)}{t} dt <\infty.
%\]
%for any $\alpha' > 1$, we may substitute $w$ with $w^{\alpha'}$ for some $\alpha' >1$ in the previous step, and thus assume that there exists a number $0<\gamma <1$ such that $w(t)/t^{\gamma}$ is non-increasing on $[0,1]$. In other words, we are now in position for which all the above lemmas and propositions in the previous subsection are applicable and thus we may freely proceed with the rest of the proof. 

\end{proof}

\section{The permanence property and $W^1$}\label{4}
\subsection{The permanence property and invisibility} 
Let $\mu$ be a positive finte Borel singular measure on $\T$. It was already mentioned in the introduction that the permanence property holds for singular inner functions in $W^1$. Our first observation in this section provides a simple way to induce the permanence property in $\B$ of singular inner functions. Let $\mu$ be a positive finite singular Borel measure on $\T$ and assume
\begin{equation}\label{positive}
  \sup \left\{ \mu(E): S_{\mu \lvert_E} \in W^1 \right\} >0.  
\end{equation}
We may pick a sequence of Borel sets $\{E_N\}_N$ with the property that each $S_{\mu\lvert_{E_N} } \in W^1$ and such that $\mu(E_N)$ converges to the supremum above. Observe that whenever $\phi, \psi$ are bounded functions in $W^1$, so is their product $\phi \psi$, and thus we may without loss of generality, assume that the $E_N \subseteq E_{N+1}$ for any $N$. Now let $E:= \cup_N E_N$ and set $\mu_0 := \mu - \mu \lvert_E$ and observe that by maximality, there exists no non-trivial $\nu_0 \leq \mu_0$ with $S_{\nu_0} \in W^1$. Measures satisfying the above condition are declared to be $W^1$ invisible. As previously observed in the introduction, each $S_{\mu_{E_N}}$ satisfies the permanence property in $\B$  and thus according to the monotonicity principle in \thref{Ppropmonprinc}, so does $S_{\mu \lvert_E}$. %We say that a positive singular measure $\mu$ on $\T$ is said to be $ W^1$ invisible if there is no singular Borel measure $0<\mu_0 \leq \mu$ for which $S_{\mu_0} \in W^1$. 
Our discussion can thus be summarized in the following result.
%We now phrase the permanence property for the singular measure $\mu \lvert_E= \mu - \mu_0$. 

\begin{prop} \thlabel{PPW11}
Let $\mu$ be a positive finite singular Borel measure on $\T$ and $\mu \lvert_E$ denote the corresponding piece in the decomposition of $\mu$ appearing in the previous paragraph. Then $S_{\mu \lvert_E}$ satisfies the permanence property in $\B$, that is, $\left[ S_{\mu \lvert_E} \right]_{\B} \cap H^\infty \subseteq S_{\mu \lvert_E } H^\infty.$
As a consequence, if $S_\mu$ is cyclic in $\B$, then $\mu$ is $ W^1$ invisible.
\end{prop}
%The proof is an immediate follows from the above discussion in conjunctions with the monotonicity principle in \thref{Ppropmonprinc}.
%\begin{proof}
%Since each $S_{\mu \lvert_{E_N}}$ belongs to $W^1$, so does the collection of reproducing kernels of $K_{S_{\mu \lvert_{E_N}}}$, which has a dense linear span in $K_{S_{\mu \lvert_{E_N}}}^1$. An application of \thref{Pprinc} with $X=W^1$ implies that each $S_{\mu \lvert_{E_N}}$ satisfies the permanence property. Now since $\mu \lvert_{E_N}$ increase up to $\mu_E$ by construction, the permanence property of $S_{\mu \lvert_E}$ immediately follows from the monotonicity principle in \thref{Ppropmonprinc}. 
 
%\end{proof}

%As a consequence, we now obtain yet another necessary condition for cyclicity of singular inner functions in $\B$. A positive singular measure $\mu$ on $\T$ is said to be $ W^1$ invisible if there is no singular Borel measure $0<\mu_0 \leq \mu$ for which $S_{\mu_0} \in W^1$.
%\begin{cor} 
%If $S_\mu$ is cyclic in $\B$, then $\mu$ is $ W^1$ invisible.
%\end{cor}

\subsection{Singular inner functions in  $W^1$}

This section is devoted to study the membership of singular inner functions in $W^1$. For any arc $I \subseteq \T$, we denote its associated Carleson square by 
$
Q_I = \left\{ z\in \D: z/|z| \in I, \, \, 1-|z| \leq |I| \right\},
$
and we let $T_I = \{z \in Q_I : 1-|z| \geq |I|/ 2 \}$ be the top-half of $Q_I$. 

\begin{lemma}\thlabel{QnotELemma} There exists a universal constant $C>0$, such that for any singular measure $\mu$ on $\T$ and any arc $I\subset \T$ with $\mu (I) =0 $, the associated singular inner function $S_\mu$ satisfies
\[
\int_{Q_I} \abs{S_{\mu}'(z)} dA(z) \leq C |I|.
\]
\end{lemma}
\begin{proof}
Let $\{I_k\}_{k \in \mathbb{Z}}$ denote the Whitney decomposition of $I$ satisfying 
\[
\text{dist}(I_{k}, \partial I) = |I_{k}| = \frac{1}{3\cdot 2^{|k|}} |I|, \qquad  k \in \mathbb{Z}.
\]
We first observe that there exists an absolute constant $C>0$ such that   
\begin{equation}\label{topparts}
\int_{Q_I \setminus \bigcup_k Q_{I_k}} |S'_{\mu} (z)| dA(z) \leq \int_{Q_I \setminus \bigcup_k Q_{I_k}} \frac{dA(z)}{1-|z|^2} \leq C |I|.
\end{equation}
% To this end, consider the collection of rectangles $\{R_k\}_k$ defined by \[
% R_0 := \{z \in Q_I: z/|z| \in I_0, \, |I_0| < 1-|z| \leq |I| \}
% \]
% and 
% \[
% R_k := \left\{z \in Q_I: z/|z| \in I_k, \qquad |I_{|k|}| < 1-|z| \leq |I_{|k|-1}| \right\}, \qquad k\neq 0.
% \]
% Notice that the collection $\{R_k\}_k$ is essentially a partition of $Q_I \setminus \bigcup_k Q_{I_k}$ (disregarding the overlap of the sides of two neighbouring rectangles, which has Lebesgue area measure zero). Now it is straightforward to check that for any $k$ we have
% \[
% \int_{R_k} \frac{dA(z)}{1-|z|} \lesssim |I_{k}|.
% \]
% Summing up over the disjoint Whitney subarcs $\{I_k\}_k$ of $J$ and using the trivial fact that $S_\mu \in \B$, we arrive at \eqref{topparts}. 
It remains to estimate the integral of $|S'_\mu |$ on $\cup_k Q_{I_k}$. To this end, for each Whitney arc $I_k$, we denote by $\xi_k$ its center and let $z_k = (1-|I_k|)\xi_k$. Note that there exists an absolute constant $c>0$ such that $c^{-1} |\zeta - z_k| \leq |\zeta - z| \leq c |\zeta - z_k|$ for any $z \in Q_{I_k}$ and any $ \zeta \in \T \setminus I$.  
%we have that $\textbf{dist}(z,\partial I) \asymp m(I_k)$. Note that for any $\zeta \T \setminus I$ we have that $|\zeta - z| \asymp |\zeta - z_k|$. 
Since $\mu(I)=0$ the Poisson extension $P(\mu)$ of $\mu$ satisfies
\[
P(\mu)(z) = \frac{1-|z|^2}{1-|z_k|^2} \int_{\T \setminus I } \frac{1-|z_k|^2}{|\zeta - z|^2} d\mu(\zeta) \asymp  (1-|z|) \frac{P(\mu)(z_k)}{|I_k|}, \qquad z \in Q_{I_k}.
\]
Now using this in conjunction with the obvious estimate 
\[
(1-|z|^2)|S'_\mu(z)| \leq 2 P(\mu)(z) \exp (-P(\mu)(z)), \qquad z \in \D , 
\]
it follows that there exists a universal constant $C>0$ such that 
\[
\int_{Q_{I_k}} |S'_\mu(z)| dA(z) \lesssim \frac{P(\mu)(z_k)}{|I_k|} \int_{Q_{I_k}} \exp \left( - C (1-|z|) \frac{P(\mu)(z_k)}{|I_k|}\right) dA(z).
\]
A straightforward computation of the integral above gives
% \[
% \int_{Q_{I_k}} \exp \left( -(1-|z|)\alpha_k \right) dA(z) % \asymp   |I_k| \int_{0}^{|I_k|} \exp(-r \alpha_k ) d r \leq
% \frac{|I_k|}{\alpha_k} = \frac{|I_k|^2 }{P(\mu)(z_k)}.
% \]
%This computation shows that on each Carleson square $Q_{I_k}$, we have the estimate
\[
\int_{Q_{I_k}} |S'_\mu (z)| dA(z) \lesssim |I_k|.
\]
Summing over $k$ completes the proof.
\end{proof}
We shall now locate the critical domain of integration in order for $S_\mu$ to belong to $W^1$, which will be convenient for our further developments. 
 %Given a closed set $E \subset \T$, let $\{I_k\}_k$ denote the connected components of $\T \setminus E$ and consider the region 
% \[
%\Omega_E := \D \setminus \cup_k Q_{I_k}.
%\]
% From the above lemma, we readily deduce that for a positive finite singular Borel measure $\mu$ supported on a closed set $E \subset \T$, the critical domain of integration for $S_\mu$ to belong to $W^1$ is the region $\Omega_E$. Below, we summarize these observations and provide yet another region related to dyadic arcs, which will be more convenient for our developments. 

\begin{cor} \thlabel{eqcondW11}
Let $\mu$ be a positive singular Borel measure supported on a closed set $E\subset \T$. Then $S_\mu$ belongs to $W^1$ if and only if 
\[
\sum_{\substack{I \in \Dy \\ I \cap E \neq \emptyset }} \int_{T_I} \abs{S'_\mu (z)} dA(z) < \infty,
\]
where $\Dy$ denotes the collection of dyadic arcs on $\T$.
\end{cor}
\begin{proof}
It is sufficient to show that the above condition implies that $S_\mu \in W^{1}$. To this end, let $\mathcal{G}$ denote the collection of maximal dyadic arcs which do not intersect $E$. Observe that we can write
\[
\int_{\D} \abs{S'_\mu (z)} dA(z) = \sum_{\substack{I \in \Dy \\ I\cap E \neq \emptyset}}\int_{T_I} \abs{S'_\mu (z)} dA(z) + \sum_{I \in \mathcal{G}} \int_{Q_I} \abs{S'_\mu (z)} dA(z)
\]
% Then 
% \[
% \bigcup_{\substack{I \in \Dy \\ I \cap E = \emptyset}}T_I = \bigcup_{I \in \mathcal{G}} Q_I.
% \]
Note that \thref{QnotELemma} gives 
\[
 \sum_{I \in \mathcal{G}} \int_{Q_I} \abs{S'_\mu (z)} dA(z) \leq C \sum_{I \in \mathcal{G}} |I| < \infty,
\]
which is enough to prove the desired claim.
\end{proof}
\noindent
As a consequence, we obtain the following sufficient condition for membership of $S_\mu$ in $W^1$. 
%Now our next results, which will play an important role in construction counter-examples to certain question is the following sufficient condition on a singular measure $\mu$, so that its associated singular inner function $S_\mu$ is a member in $W^1$.

\begin{cor}\thlabel{suffW11cond} Let $\mu$ be a positive finite singular Borel measure supported on a Beurling Carleson set of the unit circle. Then $S_\mu$ belongs to $W^1$.
\end{cor}
\begin{proof}
    It is well known that a closed set $E \subset \T$ of Lebesgue measure zero is a Beurling Carleson set if and only if
    \[
\sum_{\substack{I \in \Dy \\ I \cap E \neq \emptyset}} |I| < \infty,
\]
see \cite{borichev2017weak}. The claim now follows from Schwarz Lemma and \thref{eqcondW11}.

\end{proof}

% \begin{cor}\thlabel{suffW11cond} Let $\mu$ be a positive finite singular Borel measure carried on a set $E \subset \T$. Assume that 
% \[
% \sum_{\substack{I \in \Dy \\ I \cap E \neq \emptyset}} % \exp \left(-c \frac{\mu(I)}{|I|} \right) |I| < \infty,
%\]
% for some constant $0<c<1$, then $S_\mu$ belongs to % $W^1$.
% \end{cor}

% \begin{proof}
% Note that 
%\[
% \abs{S_\mu'(z)} \leq P(\mu)(z) \exp( -P(\mu)(z) )\frac{1}{(1-|z|)}\lesssim  \exp( -cP(\mu)(z) )\frac{1}{(1-|z|)} % \qquad z\in \D,
%  \]
% for any constant $0<c<1$. Using this in conjunction with a trivial estimate of the Poisson kernel, we get
% \[
% \int_{T_I} \abs{S_\mu'(z)} dA(z) \lesssim \exp \left( -c \frac{\mu(I)}{|I|}  \right) \int_{T_I} \frac{dA(z)}{1-|z|} \asymp \exp \left( -c \frac{\mu(I)}{|I|}  \right) |I|.
% \]
% The rest of the proof now follows from \thref{eqcondW11}.
    
% \end{proof}

% From the above corollary, we immediately see that $S_\mu$ belongs to $W^1$ for any positive finite singular Borel measure $\mu$ supported on Beurling-Carleson set.  %In what follows, we shall demonstrate that $W^1$ contains singular inner functions $S_\mu$ for $\mu$ can essentially have any modulus of continuity and any closed set of Lebesgue measure zero can support $\mu$.

\subsection{No estimate from below implies cyclicity in $\B$}
This subsection is devoted to \thref{THM:NoMOC}, which essentially asserts that no condition on the modulus of continuity on $\mu$ alone is an impediment for $S_\mu$ to be a member in $W^1$. According to the discussion in the introduction, this implies that there cannot be any estimate from below of a singular inner function which ensures it to be cyclic in $\B$.

%In fact, the proof of \thref{THM:NoMOC} and \thref{Modelneg} will be an immediate corollary of the following stronger claim which is our main agenda in this subsection.
% \begin{thm} \thlabel{Nogrowth}Let $w$ be an arbitrary majorant. Then there exists a positive finite Borel singular measure $\mu$ on $\T$ with the following properties.
% \begin{enumerate}
%     \item[(i)] There exists an absolute constant $c>0$, such that for any arc $I \subseteq \T$ 
 %   \[
 %   \mu(I) \leq c w(|I|)
 %    \]
 %    \item[(ii)] $S_\mu \in W^1$.
% \end{enumerate}
% \end{thm}
%demonstrate that for positive finite Borel singular measure $\mu$ on $\T$, there are no conditions on the modulus of continuity on $\mu$ nor on the support of $\mu$, which alone prohibit the containment of $S_{\mu} \in W^1$. Using the development in the previous subsection, we shall prove

\begin{proof}[Proof of \thref{THM:NoMOC}] 
We shall divide the proof in three different steps. 
\proofpart{1}{Construction of $\mu$} 
Let $\{n_l\}_l$ be an increasing sequence of positive integers to be specified later according to the majorant $w$. Set $\mu(\T)=1$ and consider the  subcollection $G_1 = \{I^{(1)}_j\}_j$ of every other dyadic arc of generation $n_1$. In other words $|I^{(1)}_j|= 2^{-n_1}$ and $\text{dist}(I^{(1)}_j,I^{(1)}_{j+1}) = 2^{-n_1}$ for any $j$. On these arcs, we set $\mu(I^{(1)}_j) = 2 |I^{(1)}_j|$ for each $j$, and note that
\[
\sum_j \mu(I^{(1)}_j) = 2 \sum_{j} |I^{(1)}_j| = 1.
\]
Hence $\mu$ spreads its mass precisely on the collection $G_1$. Next we consider the subcollection $G_2 = \{I^2_j\}_j$ of dyadic arcs  of generation $n_2$ contained in $\cup I^{(1)}_j$, in such a way that inside each $I^{(1)}_k$, we pick every other dyadic arc of generation $n_2$ to include in our subcollection. Thus $|I^{(2)}_j|= 2^{- n_2}$ and $\text{dist}(I^{(2)}_j,I^{(2)}_{j+1}) \geq 2^{-n_2}$ for each $j$. With this at hand, we now set $\mu(I^{(2)}_j) = 2^2 |I^{(2}_j|$ for each $j$ and observe that for any arc $I^{(1)}_k \in G_1$, we have
\[
\sum_{I^{(2)}_j \subset I^1_k} \mu(I^{(2)}_j) = 2^2 \sum_{
I^{(2)}_j \subset I^{(1)}_k} |I^{(2)}_j| = 2|I^{(1)}_k| = \mu (I^{(1)}_k).
\]
Hence inside each $I^{(1)}_k$, we again evenly distribute the mass of $\mu$ on every other dyadic arc of generation $n_2$, and denote this joint collection by $G_2 = \{I^{(2)}_j\}$. 
We proceed by induction. Assume the first $l-1$ collections of dyadic arcs have been constructed. Then we consider the collection $G_l = \{I^{(l)}_k\}$ of every other dyadic arc contained in an arc of $G_{l-1}$ and set $\mu(I^{(l)}_j) = 2^l |I^{(l)}_j|$, for each $j$. Hence for each $I^{(l-1)}_k \in G_{l-1}$ we have 
\begin{equation*} 
\sum_{I^{(l)}_j \subset I^{(l-1)}_k} \mu(I^{(l)}_j) = 2^l \sum_{
I^{(l)}_j \subset I^{(l)}_k} |I^{(l)}_j| = 2^{l-1}|I^{(l-1)}_k|= \mu (I^{(l-1)}_k). 
\end{equation*}
Note also that for each $l\geq 1$, we have 
\begin{equation}\label{Ilkprop}
    \sum_{k} |I^{(l)}_k| = 2^{-l}.
\end{equation}
Now extend $\mu$ to a Borel measure on $\T$. Note that by construction, $\mu$ is supported inside $E:=\cap_l \cup_j I^{(l)}_j$, which has Lebesgue measure zero and thus $\mu$ is indeed a positive finite singular measure. 
\proofpart{2}{Verifying condition $(i)$}
We shall now specify a choice for the sequence of positive integers $\{n_l\}_l$ so that condition $(i)$ holds. To this end, notice that since $w(t)/t \to \infty$ as $t\to 0^{+}$, for any non-negative integer $l$, we can pick $n_l$ in such a way that 
\[
\frac{w(2^{-n_l})}{2^{-n_l}} \geq   2^{l} 6.
\]
With this choice, we have
\[
\mu(I^{(l)}_j) = 2^{l}|I^{(l)}_j| = 2^{l}2^{-n_l} \leq  w(2^{-n_l}) / 6 = w(|I_j^{(l)}|) / 6.
\]
We first check that condition (i) holds for dyadic arcs $I$. We fix an arbitrary dyadic arc $I$ of length $2^{-n}$ and pick a positive integer $l$ such that $n_{l-1} \leq n < n_{l}$. We have
\[
\mu(I) = \sum_{j: I^l_j \subset I} \mu(I^{(l)}_j) = \sum_{j: I^l_j \subset I} 2^{l}\abs{I^{(l)}_j} \leq 2^{l-1} |I| \leq \frac{w(2^{- n_{l-1}})}{  2^{- n_{l-1}}} \frac{|I|}{3} \leq \frac{w(|I|)}{3},
\]
where in the last step we have used the property that $w(s)/s \leq 2 w(t)/t$ whenever $s<t$, which easily follows from the sub-additivity of majorants. Let $I \subset \T$ be an arbitrary arc and pick $n>0$ such that 
$2^{-n} < b-a \leq 2^{-(n-1)}$. Since $I$ is contained in at most three dyadic arcs of length $2^{-n}$, the estimate $(i)$ holds. 
%and $J_1 \in \Dy_n$ be the unique dyadic arc containing $a$. If $J_2, J_3 \in \Dy_n$ denotes the two consecutive arcs to the right of $J_1$, then the choice of $n$ clearly implies the containment $(a,b) \subset J_1 \cup J_2 \cup J_3$. Now the following distinct cases can occur. If $b\in J_3$, then $J_2 \subset (a,b)$, and if $b \in J_2$, then we claim that either $(J_1)_+ \subset (a,b)$ or $(J_2)_{-} \subset (a,b)$, where $I_+,I_-$ denotes the right and left of an arc $I$, respectively. Note that the former only fails if $a \in (J_1)_+$, which then by the choise of $n$ forces $b \in (J_2)_+$ and thus $(J_2)_{-} \subset (a,b)$. In any case, we conclude that any arc $I \subset \T$ contains a dyadic arc $J$ such that the union of $J$ together with its at most 2 neighbouring dyadic arcs of the same length contain $I$. This established $(i)$ with constant $C=3$.
\proofpart{3}{Verifying condition $(ii)$}
To this end, for each pair of positive integers $(k,l)$ we consider the set
\[
\Omega^{(l)}_k := \left\{z \in \D:  2^{-n_{l+1}} < 1-|z| \leq 2^{-n_l}, \, z/|z| \in I^{(l)}_k \right\}
\]
and note that according to $(iii)$ of \thref{eqcondW11} it suffices to prove that
\begin{equation}\label{A}
 \sum_{k,l} \int_{\Omega^{(l)}_k} \abs{S_\mu '(z)} dA(z) < \infty.   
\end{equation}
For $z \in \D \setminus \{0 \}$ let $I_z$ denote the arc centered at $z/|z|$ of length $1-|z|$. Observe that there exists a positive constant $c>0$ such that for $z \in \Omega^{(l)}_k$, we have that $P(\mu)(z) \geq c \mu(I_z)/|I_z| \geq c 2^{l}$. Using the formula $S_\mu '(z) = S_\mu (z) H(\mu)'(z)$, $z \in \D$, where $H(\mu)$ denotes the Herglotz transform of $\mu$, we get
\begin{equation}\label{*}
  \sum_{k,l} \int_{\Omega^{(l)}_k} \abs{S_\mu '(z)} dA(z) \leq  \sum_{k,l} e^{-c 2^l} \int_{\Omega^{(l)}_k} \abs{H(\mu) '(z)} dA(z).  
\end{equation}
The rest of the proof shall be devoted to establishing the following estimate
\begin{equation}\label{H'est}
    \int_{\Omega^{(l)}_k} \abs{H(\mu)'(z)} dA(z) \lesssim 2^{l} \, |I^{(l)}_k| = \mu(I^{(l)}_k), \quad k,l \geq 1. 
\end{equation}
Indeed, once \eqref{H'est} is established we simply use the fact that $\sum_k \mu(I_k^{(l)})= 1$ and apply the estimate \eqref{*} to deduce that 
\[
\sum_{k,l} \int_{\Omega^{(l)}_k} \abs{S_\mu '(z)} dA(z) \lesssim \sum_l e^{-c 2^l} \sum_k \mu(I_k^{(l)})  < \infty,
\]
which gives \eqref{A} and finishes the proof. So it only remains to show that estimate \eqref{H'est}  holds. For $z \in \Omega^{(l)}_k$ we write
\begin{multline*}
H(\mu)'(z) = \sum_j \int_{I^{(l+1)}_j} \frac{2\zeta d\mu(\zeta)}{(\zeta-z)^2}  = \sum_{\substack{j: I^{(l+1)}_j \subseteq I^{(l)}_k}} \int_{I^{(l+1)}_j} \frac{2\zeta d\mu(\zeta)}{(\zeta-z)^2}  + \sum_{\substack{j: I^{(l+1)}_j \cap I^{(l)}_k}= \emptyset } \int_{I^{(l+1)}_j} \frac{2\zeta d\mu(\zeta)}{(\zeta-z)^2} \\ 
= (I) + (II).
\end{multline*}
We shall first treat the term $(II)$. Note that if $z\in \Omega^{(l)}_k$ and $\zeta \in I^{(l+1)}_j$ with $I^{(l+1)}_j \cap I^{(l)}_k = \emptyset$, then $\abs{\zeta-z} \gtrsim |I^{(l)}_k|$. Then 
\[
\abs{(II)} \lesssim \sum_{\substack{j: I^{(l+1)}_j \cap I^{(l)}_k}= \emptyset }  \frac{\mu(I^{(l+1)}_j)}{\dist{I^{(l+1)}_j}{I^{(l)}_k}^2} = 2^{l+1} \sum_{\substack{j: I^{(l+1)}_j \cap I^{(l)}_k}= \emptyset } \frac{|I^{(l+1)}_j|}{\dist{I^{(l+1)}_j}{I^{(l)}_k}^2} \lesssim \frac{2^l}{|I^{(l)}_k|}.
\]
This implies 
\[
\int_{\Omega^{(l)}_k} \abs{(II)} dA \lesssim 2^l |I^{(l)}_k|,
\]
and hence we may now devote our attention to estimating $(I)$. To this end, we shall decompose it into two pieces by rewriting
\begin{multline*}
(I) = \sum_{\substack{j: I^{(l+1)}_j \subseteq I^{(l)}_k}} \int_{I^{(l+1)}_j} \frac{2\zeta}{(\zeta-z)^2} 2^{l+1}dm (\zeta) + \sum_{\substack{j: I^{(l+1)}_j \subseteq I^{(l)}_k}} \int_{I^{(l+1)}_j} \frac{2\zeta}{(\zeta-z)^2} \left( d\mu (\zeta) -2^{l+1}dm (\zeta) \right) \\ = A+B.
\end{multline*}
Denote by $\xi^{(l+1)}_j$ the center of $I^{(l+1)}_j$ and observe that since $\mu(I^{(l+1)}_j) = 2^{l+1} |I^{(l+1)}_j|$, we can add a cancellative term to $B$ and write   
\[
B = \sum_{\substack{j: I^{(l+1)}_j \subseteq I^{(l)}_k}} \int_{I^{(l+1)}_j} \left(\frac{2\zeta}{(\zeta-z)^2}-\frac{2\xi^{(l+1)}_j}{(\xi^{(l+1)}_j -z)^2} \right) \left( d\mu ( \zeta) -2^{l+1}dm (\zeta) \right).
\]
Now using the mean value Theorem together and the fact that $|\zeta-z| \gtrsim |\xi^{(l+1)}_j-z|$ whenever $\zeta \in I^{(l+1)}_j$ and $z\in \Omega_k^{(l)}$, we get 
\begin{equation}\label{|B|est}
|B| \lesssim \sum_{\substack{j: I^{(l+1)}_j \subseteq I^{(l)}_k}} \frac{\mu(I^{(l+1)}_j)+ 2^{l+1}|I^{(l+1)}_j| }{|\xi^{(l+1)}_j - z|^3} |I^{(l+1)}_j| =  \sum_{\substack{j: I^{(l+1)}_j \subseteq I^{(l)}_k}} \frac{ 2^{l+2}|I^{(l+1)}_j|^2}{|\xi^{(l+1)}_j -z|^3}.
\end{equation}
A straightforward calculation shows that 
\begin{equation} \label{IntQ3}
\int_{\Omega_k^{(l)}} \frac{dA(z)}{|\xi^{(l+1)}_j -z|^3} \lesssim \frac{1}{|I^{(l+1)}_j|}.
\end{equation}
Combining this with \eqref{|B|est}, we conclude
\begin{equation*}
\int_{\Omega_k^{(l)}}|B| dA   \lesssim 2^{l} \sum_{\substack{j: I^{(l+1)}_j \subseteq I^{(l)}_k}} |I^{(l+1)}_j| \leq 2^{l}|I^{(l)}_k|.
\end{equation*}
Hence it only remains to estimate the quantity $A$. Denote by $I+|I|$ the arc $I$ rotated by $|I|$-units and observe that by construction of the dyadic arcs $I^{(l+1)}_j$, we have 
\[
\bigcup_{\substack{j: I^{(l+1)}_j \subseteq I^{(l)}_k}} I^{(l+1)}_j \cup \left(I^{(l+1)}_j + |I^{(l+1)}_j| \right) = I^{(l)}_k.
\]
Applying the mean value Theorem as before, we get 
\[
\abs{ \int_{I^{(l+1)}_j}\frac{\zeta dm(\zeta)}{(\zeta -z)^2} - \int_{I^{(l+1)}_j + |I^{(l+1)}_j|} \frac{\zeta dm(\zeta)}{(\zeta -z)^2}  } \lesssim \frac{|I^{(l+1)}_j|^2}{|z-\xi^{(l+1)}_j|^3}, \quad z \in  \Omega_k^{(l)} , \quad j,l, k \geq 1.
\]
Hence 
\begin{multline*}
    \int_{\Omega^{(l)}_k} \abs{ A - 2^l \int_{I^{(l)}_k} \frac{2\zeta dm(\zeta)}{(\zeta-z)^2} } dA(z) \lesssim \\ 
    \int_{\Omega^{(l)}_k} 2^l  \sum_{\substack{j: I^{(l+1)}_j \subseteq I^{(l)}_k}} \abs{ \int_{I^{(l+1)}_j}\frac{\zeta dm(\zeta)}{(\zeta -z)^2} - \int_{I^{(l+1)}_j + |I^{(l+1)}_j|} \frac{\zeta dm(\zeta)}{(\zeta -z)^2}  } dA(z)  \lesssim  \\
    2^l  \sum_{\substack{j: I^{(l+1)}_j \subseteq I^{(l)}_k}} |I^{(l+1)}_j|^2 \int_{\Omega^{(l)}_k} \frac{dA(z)}{\abs{\xi^{l+1}_j -z }^3}
    \lesssim 2^l  \sum_{\substack{j: I^{(l+1)}_j \subseteq I^{(l)}_k}} |I^{(l+1)}_j| \leq 2^l |I^{l}_k|.
\end{multline*}
In the penultimate step, we have used \eqref{IntQ3}. In order to finish the proof, it just remains to verify that 
\begin{equation}\label{end}
\int_{\Omega^{(l)}_k}  \abs{\int_{I^{(l)}_k} \frac{\zeta dm(\zeta)}{(\zeta-z)^2}} dA(z) \lesssim  |I^{(l)}_k|.
\end{equation} 
To this end, let $I^{(l)}_k = (a^{(l)}_k, b^{(l)}_k)$ and observe that fixed $z \in \D$, the primitive of $\zeta/(\zeta -z)^2$ is explicitly given by $\log(\zeta -z) - z/(\zeta-z)$, which allows us to compute 
\[
\int_{I^{(l)}_k} \frac{\zeta dm(\zeta)}{(\zeta-z)^2} = \log \left( \frac{b^{(l)}_k-z}{a_k^{(l)}-z} \right) + \frac{z|I^{(l)}_k|}{(b^{(l)}_k-z)(a_k^{(l)}-z)}.
\]
Then  
\[
\abs{ \int_{I^{(l)}_k} \frac{\zeta dm(\zeta)}{(\zeta-z)^2} } \lesssim \frac{|I^{(l)}_k|}{|b^{(l)}_k-z||a_k^{(l)}-z|}, \qquad z\in \Omega^{(l)}_k,
\]
and thus
\[
\int_{\Omega^{(l)}_k}  \abs{\int_{I^{(l)}_k} \frac{\zeta dm(\zeta)}{(\zeta-z)^2}} dA(z) \lesssim  |I^{(l)}_k| \int_{\Omega_k^{(l)}} \frac{dA(z)}{|b^{(l)}_k-z||a_k^{(l)}-z|}.
\]
Since either $|z-b_k^{(l)}|\geq |I^{(l)}_k|/2$ or $|z-a_k^{(l)}|\geq |I^{(l)}_k|/2$ for any $z\in \Omega^{(l)}_k$, we deduce 
\[
\int_{\Omega_k^{(l)}} \frac{dA(z)}{|b^{(l)}_k-z||a_k^{(l)}-z|} 
% \lesssim \frac{1}{|I^{(l)}_k|} \int_{\Omega_k^{(l)}} \frac{dA(z)}{|b^{(l)}_k-z|} 
\lesssim 1.
\]
This gives \eqref{end} and finishes the proof.
% All things considered, we finally obtain 
% \[
% \int_{\Omega^{(l)}_k} |A| dA \lesssim 2^l |I^{(l)}_k|,
% % \]
% which allows us to conclude that $S_\mu$ belongs to % $W^1$.
\end{proof}

%  It is worth to remark that the estimate in \eqref{H'est} will essentially be sharp in the sense that if $L$ denotes a radial segment joining a point $a$ on the top of $\Omega^{(l)}_k$ to point $b$ on the bottom of $\Omega^{(l)}_k$, then 
% \[
% \int_{L}|H(\mu)'(\zeta)| |d\zeta| % \geq \abs{H(\mu)(a)-H(\mu)(b)} \geq % \abs{P(\mu)(a)-P(\mu)(b)} \gtrsim 2^l.
% \]
% Integrating over all possible segments and using Fubini's theorem it follows that 
% \[
% \int_{\Omega^{(l)}_k} % \abs{H(\mu)'(z)} dA(z) \gtrsim 2^{l} % \, |I^{(l)}_k|.
% \]

\subsection{The permanence property discriminates no compact set}

Here we devote our attention to the proof of \thref{THM:Nosupp}, which roughly asserts that any compact set on the unit circle of Lebesgue measure zero can support a singular measure $\mu$ for which the associated singular inner function is a member of $W^1$. This implies that no condition on the carrier set of a singular measure $\mu$ alone can describe the permanence property of $S_\mu$ in the Bloch space.

%Proof of \thref{THM:Nosupp}}

% The proof of \thref{THM:Nosupp}, which is similar in spirit as that of \thref{Nogrowth}.
%
%
%
\begin{proof}[Proof of \thref{THM:Nosupp}]
The proof is divided in three steps. 

\proofpart{1}{Constructing a covering of $E$} We first construct by induction an appropriate sequence of coverings of $E$ by dyadic arcs. Set $\G_0 := \{ \T \}$ and assume that a covering $\G_k = \{I^{(k)}_j: j\geq 1\}$ of $E$ by dyadic arcs has already been constructed  
% \[
% E \subset \bigcup_{j \geq 1} I^{(k)}_j 
% \]
and $E \cap I^{(k)}_j \neq \emptyset$ for any $j\geq 1$. For any $I \in \G_k$ we denote by $\widetilde{\G}(I)$ the collection of maximal dyadic subarcs $\widetilde{J}$ of $I$ such that at least one of the two dyadic children of $\widetilde{J}$ does not intersect $E$. Note that $\widetilde{G}(I) = \{I\}$ if one of the two dyadic children of $I \in \G_k$ does not intersect $E$. Moreover, by maximality of the collection $\widetilde{G}(I)$, for each $\widetilde{J} \in \widetilde{G}(I)$ there exists exactly one dyadic child $J$ of $\widetilde{J}$ with $J \cap E \neq \emptyset$. 
%Moreover, since $E$ has Lebesgue measure zero, for any
% any $I \in \G_k$ can be written as a disjoint union
%\begin{equation} \label{GIpartition}
%I= \bigcup_{\widetilde{J}\in \widetilde{G}(I)} \widetilde{J}.
% \end{equation}
Now consider
\[
\G(I) := \left\{J: J \, \, \text{dyadic child of some} \, \, \widetilde{J} \in \widetilde{\G}(I),  \text{ with }   J \cap E \neq \emptyset \right\}. 
\]
The set $\G_{k+1} := \bigcup_{I \in \G_k} \G(I)$ is again a covering of $E$ that will be denoted by $\G_{k+1} = \{I^{(k+1)}_j: j \geq 1 \}.$ Since $E$ is a closed set of Lebesgue measure zero, for any arc $I \in \G_k$, the union of the arcs of the collection $\widetilde{G} (I)$ covers almost every point of $I$. Hence the collection $\{\G_k\}_k$ satisfies the following packing condition: for any $I^{(k)}_l \in \G_k$ we have 
\begin{equation} \label{stackeq}
\sum_{ \substack{J \in \G_{k+1} \\ J\subset I^{(k)}_l  }} |J| = \sum_{\substack{j\geq 1 \\ I^{(k+1)}_j \subseteq I^{(k)}_l}} |I^{(k+1)}_j| = \frac{|I^{(k)}_l|}{2},\qquad k,l \geq 1.
\end{equation}
Observe that by construction $E= \cap_k \cup_{J\in \G_k} J$. Indeed, the inclusion $\subseteq$ follows from the fact that each $\G_k$ is a covering of $E$, while if $\z_0 \notin E$, then there exists a dyadic arc $I_0$ containing $\z_0$ with $I_0 \cap E = \emptyset$. Hence for every sufficiently large $k$, we must have $\z_0 \in I_0 \subset \cap_{J \in \G_k} \T \setminus J$, which establishes the reverse inclusion $\supseteq$.
\proofpart{2}{Constructing the measure $\mu$:}
We now construct $\mu$ by declaring its mass on dyadic arcs. Initially, we set $\mu(\T)=1$. Now let $I \subset \T$ be a dyadic arc and assume, by means of induction, that $\mu (I)$ has already been defined in such a way that $\mu(I)=0$ whenever $I\cap E = \emptyset$. Denote by $I_+,I_-$ the dyadic children of $I$. If both $I_+, I_-$ happens to meet $E$, then we distribute the mass evenly by declaring
\[
\frac{\mu(I_+)}{|I_+|} = \frac{\mu(I_-)}{|I_-|} = \frac{\mu(I)}{|I|}.
\]
If not, say $I_+ \cap E = \emptyset$, then we distribute all the mass to $I_-$ by setting
\[
\frac{\mu(I_-)}{|I_-|} = 2 \frac{\mu(I)}{|I|},\qquad \mu(I_+)=0.
\]
Indeed, this is consistent since 
\[
\frac{\mu(I)}{|I|} = \frac{1}{2}\frac{\mu(I_+)}{|I_+|} + \frac{1}{2}\frac{\mu(I_-)}{|I_-|}.
\]

Observe that by construction, the support of $\mu$ is equal to $E=\cap_k \cup_{J \in \G_k} J$ which has  Lebesgue measure zero. Hence $\mu$ is a non-trivial singular measure. Moreover, it is clear that for any dyadic arc $I$ with $I\cap E \neq \emptyset$, we have $\mu(I)\geq |I|$. Observe also that by construction, for any $ I^{(k)}_j \in \G_k$, we have 
\begin{equation}\label{Qest2k}
\frac{\mu(I^{(k)}_j)}{|I^{(k)}_j|} \geq 2^k, \quad j,k \geq 1.
\end{equation}
\proofpart{3}{Verifying condition $(ii)$}
It now remains to prove that the associated singular inner function $S_\mu$ belongs to $W^1$. To this end, for each $k\geq 1$ and $I^{(k)}_j \in \G_k$ consider the region
\[
\Omega^{(k)}_j := Q_{I^{(k)}_j} \setminus \bigcup_{\widetilde{J} \in \widetilde{\G}(I^{(k)}_j)} Q_{\widetilde{J}}.
\]
According to (iii) of \thref{eqcondW11}, it is sufficient to establish that
\[
\sum_{k,j} \int_{\Omega^{(k)}_j} \abs{S_\mu '} dA < \infty.
\]
Next we shall make yet another reduction. Let $I=I^{(k)}_j$ and denote by $\mu_I$ the restriction of $\mu$ to $I$. According to \thref{QnotELemma}, there exists an absolute constant $C>0$ such that  
\[
\int_{Q_I} |S'_{\mu_{\T \setminus I}}(z)| dA(z) \leq C |I|.
\]
Hence using \eqref{stackeq}, we get
\[
\sum_{k,j} \int_{\Omega^{(k)}_j}|S'_{\mu_{\T \setminus I^{(k)}_j }}(z)| dA(z) \leq C \sum_{k,j} |I^{(k)}_j| = \sum_{k\geq 1} 2^{-k} = C.
\]
With this observation at hand, it then suffices to show that 
\begin{equation}\label{estrella}
    \sum_{k,j} \int_{\Omega^{(k)}_j} |S_{\mu_{I^{(k)}_j}} '(z)| dA(z) < \infty.
\end{equation}
Using \eqref{Qest2k} in conjunction with standard estimates of Poisson kernels, there exists a universal constant $c>0$ such that 
\begin{equation}
   |S_{\mu_{I^{(k)}_j}} '(z)|  \leq e^{-c 2^k}  |H(\mu_{I^{(k)}_j})'(z)| ,\quad z \in \Omega^{(k)}_j , k,j\geq 1.  
\end{equation}
Thus \eqref{estrella} follows from the estimate 
\[
\int_{\Omega^{(k)}_j} |H(\mu_{I^{(k)}_j})'(z) | dA(z) \lesssim \mu(I^{(k)}_j) \quad j,k \geq 1.
\]
whose proof is quite similar to the proof of estimate \eqref{H'est}. We omit the details.

\end{proof}

\section{Cyclicity and invertibility}
\label{SEC:SuffcycB1}

\subsection{A sufficient condition for cyclicity}

This subsection is devoted to proving \thref{THM:SuffcycB}. A finite complex Borel measure $\nu$ on $\T$ (not necessarily singular) is said to be a \emph{Zygmund measure} if 
\[
\norm{\nu}_{*} := \, \sup \, \abs{\frac{\nu(I)}{|I|}-\frac{\nu(I')}{|I'|}} < \infty , 
\]
where the supremum is taken over all pairs of contiguous arcs $I,I'$ of the same length. %A simple but useful property of Zygmund measures is that there exists an absolute constant $C>0$ such that for any arc $I\subseteq \T$, one has
%\[
%\abs{\frac{\nu(I)}{|I|}-\frac{\nu(I/2)}{|I/2|}} \leq C \norm{\nu}_*
%\]
%where $I/2$ denotes the arc of length $|I|/2$ with the same center as $I$. It turns that Zygmund measures are intimately connected to the Bloch space. 
For the sake of future references, we shall state the following result due to Duren, Shapiro and Shields in \cite{duren1966singular} below, see also \cite{doubtsov2002symmetric}. 

\begin{lemma} \thlabel{ZygBloch}
A finite complex Borel measure $\nu$ on $\T$ is a Zygmund measure if and only if its Herglotz transform defined as 
\[
H(\nu)(z) = \int_{\T} \frac{\zeta +z}{\zeta-z} d\nu(\zeta),\qquad z\in \D,
\]
belongs to $\B$. In this case, there exists a universal constant $C>0$ such that 
\[
 \abs{H(\nu)(z) - H(\nu)(w)} \leq  C \norm{\nu}_* \beta(z,w),   \quad z,w \in \D , 
\]
where $\beta (z,w)$ denotes the hyperbolic distance on $\D$ between the points $z$ and $w$. Moreover, the following asymptotic relation holds:
\begin{equation}\label{estimacio}
P(\nu)(z) = \frac{\nu(I_z)}{|I_z|} + \mathcal{O}\left( 1 \right) \norm{\nu}_*, \qquad z \in \D . 
\end{equation}
where $I_z \subset \T$ denotes the arc centered at $z/|z|$ of length $1-|z|$.
\end{lemma}

We will also use an auxiliary result that has appeared in the setting of $\mathbb{R}^d$ in \cite{doubtsov2002symmetric}. We shall need it in the context of the unit circle where its proof is an exercise left to the reader.  

%Since we will use in the context of the unit circle, we provide its simple proof.

\begin{lemma} \thlabel{Herg'est}
Let $\nu$ be a finite Borel measure on $\T$. Then there exists an absolute constant $C>0$, such that 
\[
(1-|z|) \abs{H(\nu)'(z)} \leq C \int_0^{2 \pi} \frac{(1-|z|)}{|e^{it}-|z||^3} \abs{\nu(I(z,t)) -\nu(I(z,-t))} dt , \quad z \in \D , 
\]
where $I(z,t)$ denotes the smallest arc joining $z/|z|$ to $z e^{it} / |z|$ if $z \neq 0$ and $I(0,t)$ is the arc joining 1 and $e^{it}$.
\end{lemma}

%\begin{proof} OMIT THE PROOF AND REFER TO [7]?? Since $H'(dm)=0$, we may assume that $\nu(\T)=0$. For the sake of abbreviation, we set 
%$F_{\nu}(t) := \nu (I(1,t))$ and note that an integration by parts gives 
%\[
%H(\nu)(z) = 2iz \int_0^{2 \pi } \frac{e^{it}}{(e^{it}-z)^2} F_{\nu}(t) dt,  \qquad z\in \D.
%\]
%A change of variable followed by an integration by parts also shows that 
%\[
%\conj{H(\nu)(\conj{z})} = \int_0^{2 \pi } \frac{e^{it}+z}{e^{it}-z} d\nu(-t) = -2iz \int_0^{2 \pi } \frac{e^{it}}{(e^{it}-z)^2} F_{\nu}(-t) dt,  \qquad z\in \D.
%\]
%Note that for $0\leq r <1$, we have the identity
%\begin{equation}\label{Hergident}
%P(\nu)(r) = \frac{1}{2} \left(H(\nu)(r) + \conj{H(\nu)(r)} \right) = ir \int_0^{2 \pi } \frac{e^{it}}{(e^{it}-r)^2} \left( F_{\nu}(t)-F_{\nu}(-t) \right)dt.
%\end{equation}
%For a general point $z= re^{i\phi} \in \D$, we observe that a simple change of variable gives 
%\[
%H(\nu)(z) = \int_0^{2 \pi } \frac{e^{it}+r}{e^{it}-r} d\nu(t+\phi) = H(\nu(\cdot + \phi)(r), \qquad z=re^{i\phi} \in \D.
%\]
%Hence the identity in \eqref{Hergident} applied to the translated measures $\nu( \cdot + \phi)$ with $0\leq \phi < 2\pi$, gives 
%\begin{equation*}
%P(\nu)(z) = 
%ir \int_{\T} \frac{e^{it}}{(e^{it}-r)^2} \left( F_{\nu}(\phi+t)-F_{\nu}(\phi-t) \right)dt, \qquad z= r e^{i\phi} \in \D.
%\end{equation*}
%Differentiating with respect to $r$, we obtain 
%\begin{equation*}
%(1-|z|)|H(\nu)'(z)| \lesssim \int_{\T} \frac{(1-|z|)}{|e^{it}-|z||^3}\abs{ F_{\nu}(\phi+t)-F_{\nu}(\phi-t) }dt, \qquad z= re^{i\phi} \in \D.
%\end{equation*}
%This proves the lemma.
%\end{proof}
Given an arc $I \subset \T$ and $C>0$ let $CI$ denotethe arc with the same center as $I$ and length $C|I|$. We will also need the following technical result. 
\begin{lemma}\label{nou}
 Let $\nu$ be a positive finite Borel measure on $\T$. Assume there exists a constant $C= C(\nu)>0$ such that
    \begin{equation}\label{as1}
        \abs{\frac{\nu (I)}{|I|} - \frac{\nu (I')}{|I'|} } \leq C e^{- \nu (I)/ |I|},
    \end{equation}
for any pair of contiguous arcs $I,I' \subset \T$ of the same length. Then for any $\varepsilon >0$, there exists $\delta >0$ such that for any integer $n$ with $|n| \leq  \delta e^{\nu (I) / |I|}$ we have
\begin{equation}\label{con1}
       \abs{ \frac{\nu (2^n I)}{|2^n I|} - \frac{\nu (I)}{|I|}} \leq  \min \{  \varepsilon n , 1 \}.
   \end{equation}
% where $2^n I$ denotes the arc of length $2^n|I|$ with the same center as $I$.
  %    (b) Let $u$ be a positive harmonic function on $\D$. Assume there exists a constant $C= C(u)>0$ such that
   %   \begin{equation}\label{as3}
    %      \abs{e^{u(z)} - e^{u(w) } } \leq \beta (z,w),
     % \end{equation}
 % for any pair of points $z, w \in  \D$. Then for any $\varepsilon >0$, there exists $\delta >0$ such that
 %\begin{equation}
  %       \abs{u(w) - u(z)} \leq  \varepsilon \beta(w,z) ,\quad \text{ if    }  1 / \delta \leq % \beta(w,z) \leq % \delta e^{u(z)}. 
  %        \end{equation}

\end{lemma}
\begin{proof}
   Note that the assumption gives that $\nu$ is a Zygmund measure. Moreover \eqref{as1} gives that there exists a constant $C_1 >0$ such that
   \[
   \abs{\frac{\nu (2I)}{2I} - \frac{\nu(I)}{|I|} } \leq C_1 e^{- \max \{\nu  (2I) / |2I| , \nu (I) / |I| \} }, 
   \]
for any arc $ I \subset \T$. Using the elementary estimate $\abs{e^x - e^y} \leq e^{\max \{x,y \}} |x-y|$, $x,y \in \mathbb{R}$, we deduce
\[
\abs{ e^{\nu (2I) / |2I|} - e^{\nu(I) / |I|} } \leq C_1,
\]
for any arc $I \subset \T$. Hence for any integer $k$ we have
\begin{equation}\label{final}
\abs{ e^{\nu (2^k I) / |2^k I|} - e^{\nu(I) / |I|} } \leq C_1 |k|. 
\end{equation}
Now if $|n| \leq  (2C_1)^{-1} e^{\nu(I) / |I|}$, we obtain
\begin{equation}\label{final2}
\abs{ e^{\nu (2^n I) / |2^n I|} - e^{\nu(I) / |I|} } \leq  \frac{1}{2} e^{\nu(I) / |I|} .
\end{equation}
Hence
$$
\abs{\frac{\nu (2^n I)}{|2^n I|} - \frac{\nu(I)}{|I|}} \leq 1.
$$

Note that to check \eqref{con1} we can assume that $\nu(I) / |I|$ is large. Fix $\varepsilon >0$. Let $M=M(\varepsilon) >0$ be a large number to be fixed later. Use \eqref{final} to pick $\delta >0$ such that  $\nu(2^k I) / |2^k I| >M $ if $k$ is an integer with $|k| <  \delta e^{\nu (I) / |I|}$. Then assumption \eqref{as1} gives that
\[
\abs{ \frac{\nu (2^k I)}{|2^k I|} - \frac{\nu (2^{k-1} I)}{| 2^{k-1} I|}} \lesssim e^{-M} , \quad \text{ if } |k| \leq \delta e^{\nu (I) / |I|}. 
\]
We deduce that 
\[
\abs{ \frac{\nu (2^n I)}{|2^n I|} - \frac{\nu ( I)}{|  I|}} \leq \sum_{k=1}^{n}  
\abs{ \frac{\nu (2^k I)}{|2^k I|} - \frac{\nu ( 2^{k-1} I)}{| 2^{k-1} I|}} \lesssim e^{-M} n, \quad \text{ if } |n| \leq \delta e^{\nu (I) / |I|}. 
\]
 and taking $M \lesssim \log(1/ \varepsilon) $, estimate \eqref{con1} follows. 
\end{proof}
 % (b) Note that the assumption gives that $\sup \{|u(z)| - u(w)|: \beta(z, w) \leq 1 \} \lesssim 1 $. 
 
 % Given $\varepsilon >0$, let $M=M(\varepsilon) >0$ be a large number to be fixed later. Let $z \in \D$ with $u(z) > M$. Note that assumption \eqref{as3} gives that $e^{u(w)} \geq e^M /2$ for any $w $ in the hyperbolic disc centered at $z$ of radius $e^M /2$. Let $w,z \in \D$ with $\beta (w,z) = N > 1$. Pick points $z=z_0 , z_1, \ldots , z_N =w$ with $\beta(z_k , z_{k-1}) = 1$, $1 \leq k \leq N$. Then
% \[
% \abs{u(w) - u(z)} \leq \sum_{k=1}^{N} \abs{u(z_{k} ) - u(z_{k-1})} \lesssim  \sum_{k=1}^{N} % \abs{e^{u(z_{k} )} - e^{u(z_{k-1})}} e^{-u(z_{k-1} )} \leq N \sup_k e^{-u(z_{k-1} )}. 
% \]
% Since the last supremum is small, the proof is completed.

We are now ready to prove \thref{THM:SuffcycB}. 
\begin{proof}[Proof of \thref{THM:SuffcycB}]

It is sufficient to find a sequence of bounded analytic functions $h_n$ such that:
\begin{enumerate}
    \item[(i)] $\sup_n\norm{h_n f}_{\B} < \infty$,
    \item[(ii)]$h_n (z) f (z) \to 1$ for any $z \in \D$.
\end{enumerate}
Set $r_n := 1-2^{-n}$ and consider $h_n (z):= 1/f (r_n z)$, $z \in \D$. Note that condition $(ii)$ is trivially satisfied, hence it remains to verify $(i)$. Observe that we can write
\[
(1-|z|) \left( \frac{f(z)}{f(r_n z)} \right)' = (1-|z|) \left( \frac{f'(z)}{f(r_n z)} \right) - (1-|z|)\left( \frac{r_n f(z)f'(r_n z)}{f^2(r_n z)} \right) =: A + B.
\]
We proceed by  estimating both terms separately. 
\proofpart{1}{Estimation of B}
Set $f= \exp(-H(\nu))$ where $H(\nu)$ is the Herglotz transform of
the Herglotz-Nevanlinna measure $\nu$ of $f$. Observing that $\nu$ is necessarily a Zygmund measure, we may apply \thref{ZygBloch} to deduce
\[
(1-|r_n z|) \abs{ \frac{r_n f'(r_n z)}{f(r_n z)} }  \leq \norm{H(\nu)}_\B \lesssim \norm{\nu}_* ,  \qquad z\in \D,
\]
and thus
\begin{equation}\label{BEst}
\abs{B} \lesssim \frac{1-|z|}{1-|r_n z|} \frac{\abs{f(z)}}{\abs{f(r_n z)}} =  \frac{1-|z|}{1-|r_n z|} \exp \left( P(\nu)(r_n z) - P(\nu)(z) \right), \quad z \in \D, 
\end{equation}
where $P(\nu)$ denotes the Poisson extension of $\nu$. Now for $|z| \leq r_n$, we clearly have $\beta(z,r_n z) \lesssim 1$. An application of \thref{ZygBloch} then gives $ \abs{P(\nu)(r_n z) - P(\nu)(z)} \lesssim 1$ if $|z| \leq r_n$. 
%V\[
%\sup_{|z| \leq r_n} \abs{P(\nu)(r_n z) - P(\nu)(z)} \leq C %V\norm{\mu}_*
%V\]
and hence
\[
\sup_{|z|\leq r_n} |B| \lesssim 1.
\]
We may thus fix $z \in \D$ with $|z| \geq r_n$. Note that in order to estimate the right hand side of \eqref{BEst}, we may without loss of generality assume that $P(\nu)(r_n z)$ is large, which according to \thref{ZygBloch} is equivalent to $\nu(I_{r_n z})/|I_{r_n z}|$ being large. Doing so, we pick a large number $R>0$ for which $\nu(I_{r_n z})/|I_{r_n z}|>R$, to be specified momentarily. Observe that condition $\eqref{expzyginfcond}$, then implies that 
\begin{equation}\label{nouuu}
    \abs{ \frac{\nu(I)}{|I|}-\frac{\nu(I/2)}{|I/2|}} \leq C e^{-R}
\end{equation}
for any arc $I \subseteq I_{r_n z}$. Let $N:=N(n,z)$ denote the largest positive integer satisfying 
\[
N \leq \log \left( \frac{|I_{r_n z}|}{|I_z|} \right) = \log\left(\frac{1-|r_n z|}{1-|z|} \right).
\]
By means of iterating the estimate \eqref{nouuu} $N$ times, we obtain
\[
\abs{ \frac{\nu(I_z)}{|I_z|}-\frac{\nu(I_{r_n z})}{|I_{r_n z}|}} \leq C N e^{-R} \leq C e^{-R} \log\left(\frac{1-|r_n z|}{1-|z|} \right).
\]
Yet another application of \thref{ZygBloch} allows us to express
\[
P(\nu)(r_n z) - P(\nu)( z) = \frac{\nu(I_{r_n z})}{|I_{r_n z}|}-\frac{\nu(I_{ z})}{|I_{z}|} + \mathcal{O}(\norm{\nu}_*) 
\]
and thus we get
\[
\abs{P(\nu)(r_n z) - P(\nu)( z)} \leq C e^{-R} \log\left(\frac{1-|r_n z|}{1-|z|} \right) +  \mathcal{O}(\norm{\nu}_*).
\]
Now it is just a matter of choosing $R>0$ sufficiently large so that $C e^{-R} < 1$, which ultimately yields
\[
 |B| \lesssim \left( \frac{1-|z|}{1-|r_n z|} \right)^{1- C e^{-R}} \lesssim 1.
\]
\proofpart{2}{Estimation of A}
By definition, we readily have
\begin{equation}\label{45}
    \abs{A} \lesssim (1-|z|)\abs{H(\nu)'(z)} \exp \left( P(\nu)(r_n z) - P(\nu)(z) \right), \qquad z\in \D.
\end{equation}
As observed previously $\beta(z ,r_n z) \lesssim 1$ whenever $|z| \leq r_n$. Hence using that $\nu$ is a Zygmund measure, we have $(1-|z|) |H(\mu)' (z)| \lesssim 1$ and $| P(\nu)(r_n z) - P(\nu)(z)| \lesssim 1$ if $|z| \leq r_n$. This gives
\[
\sup_{|z|\leq r_n} \abs{A} \lesssim 1.
\]
In what follows, we shall thus fix $z \in \D$ with $|z|>r_n$ and moreover, we may also assume that $P(\nu)(r_n z)$ is sufficiently large. Denote by $I(z,t)$ the smallest arc on $\T$ joining $z/|z|$ with $z e^{it} /|z|$. An application of \thref{Herg'est} gives
\[
(1-|z|) \abs{H(\nu)'(z)} \lesssim \int_0^{2 \pi} \frac{(1-|z|)}{|e^{it}-|z||^3} \abs{\nu(I(z,t))-\nu(I(z,-t)) } dt.
\]
% In what follows, we shall need the following elementary estimate (for instance, see Lemma 2 in % \cite{anderson1991inner}):
%  \[
% \frac{1-|z|}{|e^{it}-|z| |^3} \leq C \frac{1-|z|}{\left(1-|z| +|t|/\pi \right)^{3}}, \qquad z\in \D,\, % \, |t|< \pi,
% \]
% where $C>0$ is a numerical constant. With this at hand, we can write 
% \[
% (1-|z|) \abs{H(\mu)'(z)} \lesssim \int_{-\pi}^{\pi}  \frac{(1-|z|)}{\left(1-|z| +|t|/\pi \right)^{3}} % \abs{\mu(I(z,t))-\mu(I(z,-t)) } dt.
% \]
Decomposing the integral according to the intervals $I(k)= \{ t \in [0, 2 \pi] : |t| \leq 2^k (1- |z|) \}$, we get 
\begin{multline*}
(1-|z|) \abs{H(\nu)'(z)} \lesssim (1-|z|)^{-2} \int_{I(1)} \abs{\nu(I(z,t))-\nu(I(z,-t)) } dt  \\ 
+ \sum_{k > 0} 2^{-3k} (1-|z|)^{-2} \int_{I(k+1) \setminus I(k)} \abs{\nu(I(z,t))-\nu(I(z,-t)) } dt .
\end{multline*}
Let $T= T(n, z)$ denote the integer part of $\log_2 (2^{-n} / (1-|z|))$. Observe that if $t \in I(k)$ with $k \leq T$, then $I(z,t)$ is contained in a fixed multiple of $I_{r_n z}$. This in conjunction with the assumption \eqref{expzyginfcond} yields
\[
\abs{\nu(I(z,t))-\nu(I(z,-t)) } \lesssim |t| \exp \left( - \frac{\nu(I_{r_n z})}{|I_{r_n z}|} \right).
\]
Hence we get
\[
\sum_{0\leq k < T} 2^{-3k}(1-|z|)^{-2} \int_{I(k)} \abs{\nu(I(z,t))-\nu(I(z,-t)) } dt \lesssim 
% \\ \exp \left( - \frac{\mu(I(r_n z))}{|I(r_n z)|} \right) 
% \sum_{0\leq k < T} 2^{-3k}(1-|z|)^{-2} \int_{I(k+1) } |t| dt  \lesssim 
\exp \left( - \frac{\nu(I_{r_n z}}{|I_{r_n z}|} \right). 
\]
Fix $t \in I(k+1) \setminus I(k)$ with $k>T$. Note that the arc $I_{r_n z} $ is contained in a fixed multiple of $I(z,t) \cup I(z, -t)$. According to Lemma \ref{nou}, there exists $\delta >0$ such that  
\begin{equation}\label{1}
\abs{\frac{\nu(I(z,t))}{|I(z,t)|} - \frac{\nu (I_{r_n z })}{|I_{r_n z}| } } \lesssim  1, \quad \text{ if } \log_2 (\frac{2^k (1- |z|)}{2^{-n} }) \leq \delta e^{ \frac{\nu (I_{r_n z })}{|I_{r_n z}| } }.
\end{equation}
The assumption \eqref{expzyginfcond} implies in this case
\begin{equation*}
\abs{\frac{\nu(I(z,t))}{|I(z,t)|} - \frac{\nu (I(z,-t))}{|I(z,-t)| } } \lesssim \exp \left( - \frac{\nu(I(z,t)}{|I(z,t)|} \right).    
\end{equation*}
Using estimate \eqref{1} in the right hand side term, we obtain
\[
\abs{\frac{\nu(I(z,t))}{|I(z,t)|} - \frac{\nu (I(z,-t))}{|I(z,-t)| } } \lesssim  \exp \left( - \frac{\nu (I_{r_n z })}{|I_{r_n z}| } \right), \quad \text{ if }     \log_2 (\frac{2^k (1- |z|)}{2^{-n} })\leq \delta  e^{ \frac{\nu (I_{r_n z })}{|I_{r_n z}| } }.
\]
Let $T^*$ be the largest integer such that $ T^* + \log_2 (\frac{ (1- |z|)}{2^{-n} }) \leq \delta e^{ \nu (I_{r_n z }) / |I_{r_n z}| }$. With this choice, we get 
\begin{equation*}
\sum_{k=T}^{T^*} 2^{-3k}(1-|z|)^{-2} \int_{I(k+1)} \abs{\nu(I(z,t))-\nu(I(z,-t)) } dt \lesssim 
%\\ \exp \left( - \frac{\mu(I(r_n z))}{|I(r_n z)|} \right) 
%\sum_{ k=T }^{T^*} 2^{-3k}(1-|z|)^{-2} \int_{I(k+1) } |t| dt  \lesssim 
\exp \left( - \frac{\nu(I_{r_n z})}{|I_{r_n z}|} \right). 
\end{equation*}
On the other hand, using that $\nu$ is a Zygmund measure we arrive at
\begin{equation*}
\sum_{k > T^*} 2^{-3k}(1-|z|)^{-2} \int_{I(k+1)} \abs{\nu(I(z,t))-\nu(I(z,-t)) } dt \lesssim 
%\\ 
%\sum_{ k > T^*} 2^{-3k}(1-|z|)^{-2} \int_{I(k+1) } |t| dt  \lesssim 
\exp \left( - \frac{\nu(I_{r_n z})}{|I_{r_n z}|} \right). 
\end{equation*}

Finally, combining all terms in the sum, we obtain
\[
(1-|z|)|H'(\nu)(z)| \lesssim \exp \left(- \frac{\nu(I_{r_n z})}{|I_{r_n z}|} \right), \qquad |z|>r_n.
\]
With these estimates at hand, we may now return back to \eqref{45}, deducing that \\ $ |A| \lesssim \exp \left( -P(\nu)(z) \right)  \leq 1$. This completes the proof.
\end{proof}

\subsection{A characterization of invertibility in the Bloch space}

\begin{proof}[Proof of \thref{THM:Invertibility B}]
In order to prove the sufficiency, note that by means of writing $(1/f)' = H(\nu)' /f$, it suffices to show that $(1-|z|^2) |H(\nu)' (z)| \lesssim |f (z)|$, $ z \in \D.$
%\[
%(1-|z|^2) |H(\nu)' (z)| \lesssim |f (z)|, \qquad z \in \D.
%\]
Again, condition \eqref{expzyginfcond2} readily implies that $\nu$ is a Zygmund measure. In view of the estimate \eqref{estimacio} in Lemma \ref{ZygBloch}, it suffices to prove that
\begin{equation}\label{objectiu}
    (1- |z|^2) |H(\nu)' (z)| \lesssim e^{- \nu (I_z ) / |I_z|} , \quad z \in \D. 
\end{equation}
For $z \in \D$ and $n \in \mathbb{Z}$, consider the intervals $I_n = \{t \in [0, 2 \pi] : 2^{n-1} (1-|z|)  \leq t < 2^{n} (1-|z|) \}$. According to \thref{Herg'est} we have
\[
(1- |z|^2) |H(\nu)' (z)| \lesssim \sum_{n \in \mathbb{Z}} \int_{I_n} \frac{1-|z|}{|e^{it} -  |z||^3} \abs{\nu (I(z,t)) - \nu (I(z,t)) } dt, \qquad z \in \D.
\]
We now estimate each term in the sum. Assumption \eqref{expzyginfcond2} implies
\begin{equation}\label{sjor}
   \abs{\nu ( I(z,t)) - \nu (I(z,-t)) } \leq C |t| e^{-\nu (I_t (z)) / |I_t (z)|  }, 
\end{equation}
where as before, $I_t (z)$ denotes the arc $I_t (z) = I(z,t) \cup I(z, -t)$. Let $\varepsilon >0$ be a small number to be fixed later. Applying Lemma \ref{nou}, we obtain a number $\delta >0$ such that
\begin{equation}\label{30}
\abs{\frac{\nu (I_t (z))}{|I_t (z)|} - \frac{\nu (I_z)}{|I_z|} } \leq \varepsilon n, \quad t \in I_n , \quad |n| \leq T := \delta e^{\nu (I_z) / |I_z|}.     
\end{equation}
It thus follows from \eqref{sjor} and \eqref{30} that 
\[
\abs{\frac{\nu (I(z, t))}{t} - \frac{\nu (I(z, -t)}{t} }  \lesssim e^{\varepsilon n} e^{- \nu (I_z) / |I_z| } , \quad t \in I_n , \quad |n| \leq T. 
\]
Now if $ 0 < \varepsilon < \log 2$, we deduce that
\[
\sum_{n : |n| \leq T} \int_{I_n} \frac{1-|z|}{|e^{it} -  |z||^3} \abs{\nu (I(z,t)) - \nu (I(z,-t)) } dt 
%\lesssim e^{- \nu (I_z) / |I_z|} \sum_{n : |n| \leq T} 2^{-n} e^{\varepsilon n}
\lesssim e^{- \nu (I_z) / |I_z|} .
\]
The tail of the sum is estimated using the fact that $\nu$ is a Zygmund measure. In fact, one obtains  
\[
\sum_{n : |n| > T} \int_{I_n} \frac{1-|z|}{|e^{it} -  |z||^3} \abs{\nu (I(z,t)) - \nu ( I(z,t)) } dt 
\lesssim 2^{-T} \leq C(\delta) e^{-\nu (I_z) / |I_z|} ,  
\]
where $C(\delta) >0$ is a constant depending on $\delta$. This proves the estimate \eqref{objectiu} and therefore completes the proof of the sufficiency. We now turn our attention to the necessity part of the proof. Rewrite $f' = H(\nu)' f$, we see that our assumption reads
\begin{equation}\label{invert}
(1-|z|^2) |H(\nu)' (z)| \lesssim e^{-P(\nu) (z)} , \quad z \in \D.     
\end{equation}
It turns out that condition \eqref{invert} implies that
\begin{equation}\label{estimacio2}
P(\nu)(z) = \frac{\nu(I_z)}{|I_z|} + \mathcal{O}\left( 1 \right) e^{- \nu (I_z) /  |I_z|} , \quad z \in \D.   
\end{equation}
This is completely analogous to the estimate \eqref{estimacio} 
stated in Lemma \ref{ZygBloch} which holds for Zygmund measures. Actually one can prove \eqref{estimacio2} mimicking the proof of \eqref{estimacio} given in \cite{doubtsov2002symmetric}. The details are omitted. Note that \eqref{invert} gives that $H(\nu) \in \B$. According to Lemma \ref{ZygBloch}, $\nu$ is a Zygmund measure. Using this, we see that the estimate \eqref{invert} implies
\[
\abs{P(\nu)(z_I) - P(\nu)(z_{I'})}  \leq \abs{H(\nu) ( z_I) - H(\nu) ( z_{I'}) } \lesssim e^{-\nu (I) / |I|},
\]
for any pair of contiguous arcs $I, I' \subset \T$ of the same length. % In fact, mimicking the proof of \eqref{estimacio} given in \cite{doubtsov2002symmetric} one may actually deduce that the estimate \eqref{invert} implies the following improved asymptotic relation:
% \begin{equation}\label{estimacio2}
% P(\nu)(z) = \frac{\nu(I_z)}{|I_z|} + \mathcal{O}\left( 1 % \right) e^{- \nu (I_z) /  |I_z|} , \quad z \in \D.   
% \end{equation}
This in conjunction with \eqref{estimacio2} then shows that condition \eqref{expzyginfcond2} holds, which finishes the proof. 
\end{proof}
% Without loss of generality we can assume that $ \mu (I(z)) / |I(z)|$ is large. If $2^{n-1} (1- |z|) < t < 2^n (1-|z|) $ with $|n \lesssim \mu (I(z)) / |I(z)|$, we have
% \[
% \abs{\frac{\mu(I(z,t))}{t} - \frac{\mu(I(z,-t))}{t} } \leq e^{\varepsilon n } e^{- \mu (I(z)) / % |I(z)|}. 
%\]
%Apply \ref{Herg'est} to obtain 
%\[
% (1-|z|) \abs{H(\nu)'(z)} \leq C \int_0^{2 \pi} \frac{(1-|z|)}{|e^{it}-|z||^3} \abs{\nu(I(z,t)) -% \nu(I(z,-t))} dt , \quad z \in \D 
% \]

%\end{proof}

% \subsection{A characterization of invertibility in $\B$}
% Here shall devote our efforts into establishing \thref{THM:Invertibility B}. The following technical lemma will be our crucial tool.

% \begin{proof}[Proof of \thref{THM:Invertibility B}]

\subsection{Invertibility does not imply cyclicity in the Bloch space}
This section is devoted to a self-contained proof of \thref{THM:CycvsInv}.
\begin{proof}[Proof of \thref{THM:CycvsInv}]
Let $E$ be a sequence of complex numbers containing the origin such that both $\mathbb{C} \setminus E$ and $\mathbb{C} \setminus E^{-1}$ are open sets which do not contain arbitrarily large discs, that is, there exists $R_0 >0$ such that no disc of radius $R_0$ is contained in either $\mathbb{C} \setminus E$ or $\mathbb{C} \setminus E^{-1}$. Here $E^{-1} = \{w \in \mathbb{C} : 1/w \in E \}$. Let $f : \D \rightarrow \mathbb{C} \setminus E $ be an analytic universal covering map. According to Bloch's Theorem, we have that both $f$ and $1/f$ belong to $\B$. Next we shall show that $f$ is not cyclic in $\B$.

Assume $f$ has non-tangential limit, say $L$, at a point of the unit circle. Since $f$ is a covering map, $L$ must be either infinity or a point in $E$. Since $E$ is countable, Privalov's Theorem gives that $f$ can only have finite non-tangential limit at a set of Lebesgue measure zero of points of the unit circle. By Plessner's Theorem (for instance, see p. 205 in \cite{garnett2005harmonic}), its non-tangential cluster set at almost every point of the unit circle is the whole extended plane, that is, 
\begin{equation}\label{plessner}
\overline{    {\cap}_{0<r<1} f(\Gamma_r (\xi)) } = \mathbb{C} \cup \{\infty \}, \quad m\text{- a.e. } \xi \in \T .  
\end{equation}
Here $\Gamma_r (\xi) = \{z \in \D : |z-\xi| < M(1-|z|), |z| > r \}$ denotes the truncated Stolz angle of some fixed aperture $M>1$ with vertex at the point $\xi \in \T$. 

Next we will show that $\A (f) = \{f h \in \B: h \in H^\infty \}$ is weak-star closed in $\B$. Since $\A(f)$ is convex, it suffices by the Krein-Smulian Theorem to show that it is weak-star sequentially closed. To this end, let $h_n \in H^\infty$ such that $fh_n$ tends to $F$ weak-star in $\B$. Then there exists a constant $C>0$ such that 
\begin{equation}\label{last}
    \sup_{z \in \D} \abs{(1-|z|^2) ( f'(z) h_n (z) + f(z) h'_n (z) )} \leq C. 
\end{equation}
Fix $a \in \mathbf{C} \setminus E $ and use \eqref{plessner} to show that for almost every $\xi \in \T$ there exist points $z_k (\xi) \in \D$ tending non-tangentially to $\xi$ such that $f(z_k (\xi)) \to a$. Since $(1-|z|^2) |f'(z)| \geq \dist{f(z)}{E}$, $z \in \D$ (see \cite{beardon1978poincare}), we deduce that there exists a constant $c>0$ such that
\begin{equation}\label{inferior}
    \liminf_{k \to \infty} (1-|z_k (\xi)|^2) |f'(z_k (\xi))| > c
\end{equation}
Note that Bloch's Theorem shows that if an analytic function $F$ in the unit disc has non-tangential limit at a given point $\zeta \in \T$, then $(1-|z|) |F'(z)| \to 0$ as $z \in \D$ tends non-tangentially to $\zeta$. Since $h_n$ has non-tangential limit along almost every radius, for each integer $n$ we have
\[
\lim (1-|z|)|h'_n (z)|=0, \quad m\text{- a.e. } \xi \in \T,
\]
where the limit is taken as $z \in \D$ tends to $\xi$ non-tangentially. Applying \eqref{last} and \eqref{inferior} we deduce that
\[
\limsup_{k \to \infty} | h_n (z_k (\xi))| \leq C, \quad m\text{- a.e. } \xi \in \T , \quad n=1,2,\ldots .
\]
Hence $\|h_n\|_\infty \leq C $, $n=1,2,\ldots$ and we may extract a subsequence of $\{h_n \}$ which converges pointwise in $\D$ to a function $h \in H^\infty$ and thus $F=fh$. We conclude that $\A (f)$ is weak-star closed in $\B$. Now $\left[ f\right]_{\B} \subseteq \A (f)$ and it is clear that $1$ does not belong to $\A (f)$, hence we conclude that $f$ is not cyclic in $\B$.

\end{proof}

\subsection{A problem of weak-star sequential closure in the Bloch space}
%\subsection{Weak-star closed $M_z$-invariant subspaces in $\B$}
%Here we shall construct a function $f\in \B$, such that $1/f$ belongs to $\B$, but the set of all weak-star sequential limits in $\B$ of functions of the type $f(z)Q(z)$, where $Q$ is an analytic polynomial, does not contain all of $\B$. This will prove \thref{THM:Seqcyc}. The following result is the key observation towards our goal.
In the proof of \thref{THM:Seqcyc} we will need to construct functions in $\B$ which enjoy maximal radial growth at considerable large subset of $\D$.

\begin{prop}\thlabel{unbddnoncyc} There exists $f \in \text{BMOA}\subset \B$ satisfying the following properties:
\begin{enumerate}
    \item[(i)] The function $1/f$ belongs to $H^\infty$.
    \item[(ii)] For some $\delta >0$, the set $E(f) := \left\{ z\in \D: \abs{f(z)} \geq \delta \log (\frac{1}{1-|z|}) \right\}$
accumulates everywhere on the unit circle $\T$, that is, $\T \subset \overline{E(f)}$. 
\end{enumerate}

\end{prop}
\noindent
Observe that such an $f$ must necessarily be outer, hence it is cyclic $\B$. The set of points for which the Bloch function $f$ attains its maximal growth are located in the set $E(f)$. In the context of cyclic vectors in growth spaces, Hedenmalm and Borichev observed, roughly speaking that cyclic functions therein cannot grow "maximally" on too "massive" sets. See \cite{borichev1995cyclicity}. Our construction below is deeply inspired from this idea. Taking \thref{unbddnoncyc} for granted, we first deduce \thref{THM:Seqcyc}.

\begin{proof}[Proof of \thref{THM:Seqcyc}]
Let $f \in BMOA$ be given by Proposition \ref{unbddnoncyc}. Let $\{Q_n\}_n$ be a sequence of analytic polynomials with $\sup_n \norm{fQ_n}_{\B}< \infty$. It follows that
\[
\sup_{n} \sup_{z\in \D} \log^{-1}\left( \frac{e}{1-|z|}\right)\abs{f(z)Q_n(z)} < \infty.
\]
However, this implies
\[
\sup_n \sup_{z\in E(f)} \abs{Q_n(z)} < \infty.
\]
Now the assumption on $E(f)$ in conjunction with the maximum principle implies that \\ $\sup_{n} \norm{Q_n}_{H^\infty} < \infty.$ By Helly's selection theorem, we can extract a subsequence $\{Q_{n_k}\}_k$ which converges pointwise on $\D$ to some function $h\in H^\infty$. Hence any weak star sequential limit of $\{f Q_n\}$ is of the form $fh$ with $h \in H^\infty$. But functions of that form have finite radial limits at $m$-a.e on $\T$, hence the set $\{fh \in \B : h \in H^\infty \}$ is a proper subspace of $\B$, which completes the proof. 

\end{proof}

The proof of Proposition \ref{unbddnoncyc} is based on the following construction.

\begin{lemma} \thlabel{ACmeasure}
Let $0<\delta <1$. Then there exists a positive finite Borel measure $\nu$ on $\T$ satisfying the following property: For any arc $I\subset \T$ there exists a dyadic subarc $J\subset I$, such that $\nu(J) \geq  |J|^{1-\delta}.$
\end{lemma}

\begin{proof}
Fix $0<\eta <1$ to be specifiec later. We shall define the measure $\nu$ by declaring its mass on dyadic arcs. First we set $\nu(\T)=1$. 
% and let $I_+,I_-$ denote the right and left dyadic children of $I:= \T$, respectively, and declare that 
% \[
% \frac{\nu(I_{\pm})}{|I_{\pm}|} = (1 \pm \eta) \frac{\nu(\T)}{|\T|} = (1\pm \eta).
% \]
Inductively, assume that $\nu(J)$ has already been defined for some dyadic arc $J\subset \T$. We denote by $J_+,J_-$ the right and left dyadic children of $J$, respectively. We shall assign a larger portion of the available mass to the right child than to the left by declaring 
\[
\frac{\nu(J_\pm)}{|J_\pm|} = (1\pm \eta)\frac{\nu(J)}{|J|}.
\]
% Indeed, this is coherent with the martingale property
% \[
% \frac{\nu(J)}{|J|} = \frac{1}{2}\frac{\nu(J_+)}{|J_+|} + \frac{1}{2} \frac{\nu(J_-)}{|J_-|}.
% \]
Now it is a standard fact in measure theory that this construction gives rise to a uniquely defined positive finite Borel measure $\nu$ on $\T$. Hence it suffices to verify that $\nu$ has the required property. We may without loss of generality assume that $I$ is a dyadic arc. Now let $J_n \subset I$ be the rightmost dyadic subarc of $I$ with $|J_n|= 2^{-n}|I|$. Then by construction, we have 
\[
\frac{\nu(J_n)}{|J_n|} = (1+\eta)^n \frac{\nu(I)}{|I|} \geq \frac{1}{|J_n|^{\delta}} = \frac{2^{\delta n}}{|I|^{\delta}}
\]
whenever $(1+\eta)^n/ 2^{n\delta} \geq |I|^{1-\delta}/\nu(I)$. However, choosing $0<\eta<1$ so that $2^{\delta} < 1+\eta$, the above inequality will certainly hold for sufficiently large $n$. This finishes the proof.
\end{proof}

\begin{proof}[Proof of \thref{unbddnoncyc}] 
Let $\nu$ be the positive finite Borel measure given by \thref{ACmeasure}. Consider an analytic selfmap $b= b_\nu$ of $\D$ defined via the Herglotz transform of $\nu$:
\[
\frac{1+b (z)}{1-b (z)}  = \int_{\T} \frac{\zeta + z}{\zeta - z} d\nu(\zeta),\qquad z\in \D.
\]
In other words, $\nu$ is the so-called Aleksandrov-Clark measure for $b$. Consider the function 
\[
f(z) = \log \left( \frac{e}{1-b (z)} \right), \qquad z \in \D.
\]
Note that since $f$ has bounded imaginary part, it follows that $f$ belongs to $\text{BMOA} \subset \B$. Moreover, the trivial fact that $\abs{1-b(z)} < 2$ for $z\in \D$, implies that $1/f \in H^\infty$. It remains to show that the  closure of the corresponding set $E(f)$ contains the entire unit circle. Indeed, the reason for our choice of $f$ stems from the observation that the set $E(f)$ contains
\[
\widetilde{E}(b):= \left\{z\in \D: \Re \left(\frac{1+b (z)}{1-b (z)} \right) \geq (1-|z|)^{-\delta} \right\}, 
\]
thus it suffices to check that the closure of the smaller set $\widetilde{E}(b)$ contains $\T$. To this end, fix a small $\varepsilon >0$ and let $\zeta \in \T$ be an arbitrary point. If $0<\delta'<1$ is fixed, then an application of \thref{ACmeasure} shows that for any arc $I \subset \T$ containing $\zeta$ with length at most $\varepsilon /2$, there exists a small dyadic arc $J \subset I$ (not necessarily containing $\zeta$), such that $\nu(J) \geq |J|^{1 - \delta'}$. Let $\xi_J$ denote the center of the arc $J$ and $z_J := (1-|J|)\xi_J$, and observe that a trivial estimate of the Poisson kernel implies that 
\[
\Re \left(\frac{1+b (z_J)}{1-b (z_J)} \right) = P(\nu)(z_J) \geq c \frac{\nu(J)}{|J|} \geq \frac{c}{|J|^{\delta'}} = c(1-|z_J|)^{-\delta'} \geq (1-|z_J|)^{-\delta},
\]
where $c>0$ is an absolute constant and $0<\delta<\delta'$ sufficiently small. But this shows that $z_J \in \widetilde{E}(b)$ with $\abs{z_J - \zeta} \leq 2|I| < \varepsilon$. Hence it follows that for sufficiently small $\delta>0$, the corresponding set $E(f)$ accumulates to any point $\zeta \in \T$, thus we conclude the proof of \thref{unbddnoncyc}.
\end{proof}

%Can we use \thref{THM:SuffcycB} to provide a simpler example of a cyclic inner function in $\B$?

\bibliographystyle{siam}
\bibliography{mybib}

\Addresses

\end{document}